\documentclass[reqno]{amsart}

\usepackage[english]{babel}
\usepackage[T1]{fontenc}
\usepackage[utf8]{inputenc}
\usepackage{amsthm}
\usepackage{amsmath}
\usepackage{amssymb}
\usepackage{wasysym}
\usepackage{ragged2e}
\usepackage{color}
\usepackage{mathrsfs}
\usepackage[all,cmtip]{xy}
\usepackage{enumerate}
\usepackage{hyperref}
\theoremstyle{remark}

\theoremstyle{plain}
\newtheorem{defin}{Definition}[section]

\newtheorem{prop}[defin]{Proposition}
\newtheorem{thm}[defin]{Theorem}
\newtheorem{cor}[defin]{Corollary}
\newtheorem{lm}[defin]{Lemma}
\theoremstyle{remark}
\newtheorem{rmk}{Remark}
\theoremstyle{remark}
\newtheorem*{ex}{Example}
\newcommand{\prodscal}[2]{\left\langle#1,#2\right\rangle}
\usepackage[utf8]{inputenc}

\newcommand{\BB}{\mathbb{B}}
\newcommand{\CC}{\mathbb{C}}

\newcommand{\GG}{\mathbb{G}}
\newcommand{\HH}{\mathbb{H}}

\newcommand{\sA}{\mathcal{A}}

\newcommand{\sD}{\mathcal{D}}
\newcommand{\sE}{\mathcal{E}}

\newcommand{\sL}{\mathcal{L}}

\newcommand{\sV}{\mathcal{V}}

\DeclareMathOperator{\Hom}{Hom}

\newcommand{\btimes}{\hat{\otimes}}

  {\par\begin{tabular}{rcl}}%
  {\end{tabular}\par}

\makeatletter
\newcommand*\slot{\mathpalette\slot@{.5}}
\newcommand*\slot@[2]{\mathbin{\vcenter{\hbox{\scalebox{#2}{$\m@th#1\bullet$}}}}}
\makeatother

\DeclareMathOperator{\Ind}{Ind}

\DeclareMathOperator{\id}{id}

\title{Explicit Rieffel induction module for quantum groups}
\author{Damien Rivet}
\address{Université Clermont Auvergne, CNRS, LMBP, F-63000 Clermont-Ferrand, France}
\email{damien.rvt@gmail.com}

\subjclass[2010]{20G42, 16T05, 46L65, 46L51}
\keywords{Induction, Quantum groups, bornological algebras, algebraic quantum groups, locally compact quantum groups, semisimple quantum groups} 

\begin{document}

\maketitle
\begin{abstract}
For $\mathbb{G}$ an algebraic (or more generally, a bornological) quantum group and $\mathbb{B}$ a closed quantum subgroup of $\mathbb{G}$, we build in this paper an induction module by explicitly defining, on the convolution algebra of $\mathbb{G}$, an inner product which takes its value in the convolution algebra of $\mathbb{B}$, as in the original approach of Rieffel. In this context, we study the link with the induction functor defined by Vaes. In the last part we illustrate our result with parabolic induction of complex semi-simple quantum groups. We first show that our induction functor coincides with the one already defined in the case of parabolic induction. Then we use the tools developed in this paper to give a geometric interpretation to the parabolic induction functor, following the approach suggested by Clare in the classical case.
\end{abstract}
\section{Introduction}
Let $G$ be a locally compact group and $B$ a closed subgroup of $G$. One can build unitary representations of $G$ from those of $B$ with the unitary induction procedure due to Mackey \cite{Mackey}, who also developed the concept of imprimitivity. Rieffel \cite{Rieffel} gave an alternative and more general formulation in the $C^*$-algebraic setting by using $C^*$-Hilbert modules. In short, there exists a Hilbert $C^*_u(B)$-module $\mathcal{E}(G)$, with a left representation of $C^*_u(G)$, such that for a unitary representation of $B$ on any Hilbert module $V$, $\mathcal{E}(G)\otimes_{C^*_u(B)}V$ corresponds to the induced unitary representation of $G$.

In the case where $\mathbb{G}$ is a locally compact quantum group and $\mathbb{B}$ a closed quantum subgroup, induction procedures have been developed by  Kustermans \cite{Kustermans} and Vaes \cite{Vaes}. Vaes was able to formulate this in a wide framework and to state imprimitivity theorems.
In this paper we develop an approach closer to the original one of Rieffel, by directly defining the induction module $\mathcal{E}(\GG)$.

The main difficulty is that, unless $\mathbb{B}$ is also an open subgroup of $\mathbb{G}$ (the case treated in \cite{Kalantar}), we don't have an inclusion of $C^*_u(\mathbb{B})$ into $C^*_u(\mathbb{G})$ and so it is not possible to define a conditional expectation from $C^*_u(\mathbb{G})$ to $C^*_u(\mathbb{B})$. In his original paper Rieffel avoided this issue by considering the convolution algebra $C_c(G)$ of compactly supported functions on $G$, instead of the full space $C^*(G)$. Then he defined a weak conditional expectation $C_c(G)\rightarrow C_c(B)$.

The bornological setting for quantum groups developed by C. Voigt \cite{Born} allows us to consider algebras with similar properties and then to define an analogue of the weak conditional expectation. One of the main goals of this paper is to show that, in this particular case, the induction functor we obtain is the same as the one defined by Vaes. We thus get a more direct way to compute induced representations for regular bornological quantum groups and to apply the powerful imprimitivity theorems. We remark that the class of bornological quantum groups is a large subclass of locally compact quantum groups. The only known obstruction to being bornological is non-regularity and regularity is already a necessary condition in Vaes imprimitivity theorem. The class of bornological quantum groups includes compact quantum groups, algebraic quantum groups \cite{Vandaele} (and in particular complex semi-simple quantum groups) and classical locally compact groups. We remark that our results are new even in the case of algebraic quantum groups.

In the last part, we illustrate the general construction with the example of principal series representations of a semi-simple complex quantum group $G_q$ \cite{yuncken-voigt}. First, we show that our induction functor coincides with the one already defined in the case of parabolic induction \cite{Arano,yuncken-voigt}. After that, in analogy to the classical case in \cite{CCH}, we build a module using a $G_q$-space $G_q/N_q$, which implements the parabolic induction. The notation $G_q/N_q$ is meant to suggest a homogeneous space with respect to a quantum analogue of the classical unipotent subgroup although we do not actually use any such subgroup in its definition. As well as giving a noncommutative geometry perspective on the parabolic induction functor for $G_q$ (similar to the approach of \cite{Gelf} in the case $G_q=SL_q(2,\CC)$), we can thus provide a new description of the structure of the reduced $C^*$-algebra $C^*_r(G_q)$ following the results of \cite{Plancherel}\cite{MonkV}.

In our proofs, we will make certain assumptions on the bornological quantum groups we consider.  To begin with, we will assume that the scaling constants of $\mathbb{G}$ and $\mathbb{B}$ are $1$, as well as the scaling constant associated to the restriction $\pi(\delta_{\mathbb{G}})$ of the modular element of $\mathbb{G}$ to $\mathbb{B}$, see Remark \ref{rmk:hyp}. These hypotheses are almost certainly unnecessary, but they greatly simplify the constructions and in any case, at present, we don't know any examples of bornological quantum groups for which they do not hold.  More significantly, we will suppose that the closed quantum subgroup is amenable, so that $C^*_u(\mathbb{B})=C^*_r(\mathbb{B})$.  
This hypothesis is made for a technical reason, namely to prove positivity of a the $C^*(\mathbb{B})$-valued inner product on the Rieffel induction module. Again, we suspect this is not necessary, but we don't currently have a proof of positivity in the general case.

\section{Bornological quantum groups}
\label{sec:BQG}

Bornological quantum groups, defined by Voigt \cite{Born}, are a generalization of algebraic quantum groups introduced by Van Daele \cite{Vandaele} where most of the interesting properties stay valid.
In this section, we recall the definition of a bornological quantum group and state some of the important properties.  Many of the properties which we shall need in this paper, particularly concerning the relationship between Voigt's theory of bornological quantum groups \cite{Born} and Kustermans and Vaes' theory of locally compact quantum groups \cite{K&V}, are analogues of well-known properties of algebraic quantum groups which had not previously been proven in the bornological context.  This void is filled by the article \cite{RY}.  We will make much use of that article as well as Voigt's original article.  But in order to keep this article self-contained, we will summarize the necessary results here.

For the reader interested only in algebraic quantum groups, it is possible to read this entire article replacing ``bornological quantum groups'' with ``algebraic quantum groups''.  Then bounded maps become arbitrary maps, bornological tensor products becoming algebraic tensor products, and so on, see Example \ref{ex:AQG}.  In this case, the prerequisite results are all well-known from the works of Van Daele, Kustermans and De Commer \cite{Vandaele,Kus2, kus2003,commer}.

We begin with the basic definitions of bornological vector spaces.  For more details see \cite{hogbe, MEY}.

A bornology on a vector space $V$ is a covering family $\mathcal{B}$ of subsets of $V$, called bounded sets, which is stable under taking subsets and finite unions, and such that the vector space operations map bounded sets to bounded sets. 
The guiding example is the set of bounded subsets of a topological vector space.
We will always impose the convexity condition that $\mathcal{B}$ is stable under taking balanced convex hulls, often called disks.  Each bounded disk $D$ gives rise to a seminorm on its linear span $V_D = \mathrm{span(D)}$ for which $\overline{D}$ is the unit ball.  Then $V$ is called complete if every bounded set is contained in some bounded disk $D$ for which $V_D$ is a Banach space.

A map between bornological vector spaces is called bounded if it maps bounded sets to bounded sets.  If $V$ and $W$ are complete bornological vector spaces, then there exists a bornological tensor product $V\hat\otimes W$ which is a universal target for bounded bilnear maps from $V\times W$.  One of the nice features of bornological vector spaces is the Hom-tensor adjunction
\[
  \mathrm{Hom}(V\hat\otimes W, X) \cong \mathrm{Hom}(V,\mathrm{Hom}(W,X))
\] 
where $\mathrm{Hom}$ denotes the bounded linear maps.
In order to avoid pathologies, one should add the \emph{approximation property}, which says that the identity map of $V$ can be approximated uniformly on compact subsets by finite-rank operators, see \cite{MEY2}.  This technical condition will be true of all our examples, and we will not make mention of it.

A bornological algebra is a complete bornological vector space $\mathcal{A}$ equipped with a bounded algebra product $\mathcal{\mathcal{A}} \times \mathcal{A} \to \mathcal{A}$.  It therefore extends to the bornological tensor product $\mathcal{A}\hat\otimes \mathcal{A} \to \mathcal{A}$.  It is called essential if the product induces a bornological isomorphism $\mathcal{A}\hat\otimes_\mathcal{A} \mathcal{A} \cong \mathcal{A}$.

The space of two-sided multipliers of a bornological algebra $\mathcal{A}$ is denoted $\mathcal{M}(\mathcal{A})$.  See \cite[Section 3]{Born} for the precise definition.  A bounded algebra morphism $f:\mathcal{A} \to \mathcal{M}(\mathcal{B})$ is called essential if it induces bornological isomorphisms $\mathcal{A} \hat\otimes_\mathcal{A} \mathcal{B} \cong \mathcal{B} \cong \mathcal{B}\hat\otimes_\mathcal{A} \mathcal{A}$.  In this case $f$ extends uniquely to the multiplier algebra of $\mathcal{A}$.

We modify Voigt's original definition of a bornological quantum group by adding a $*$-structure.  A $*$-structure on a bornological algebra is a bounded involutive anti-automorphism on $\sA$.  For more details, see \cite{RY}.

\begin{ex}
	\label{ex:AQG}
	Any vector space $V$ can be equipped with the $\mathfrak{fine}$ bornology, for which the bounded subsets are precisely the bounded subsets (in the usual sense) of finite dimensional subspaces of $V$.  Any linear map from $V$ to a bornological vector space $W$ is bounded with respect to the fine bornology, and the bornological tensor product $V\otimes V$ is just the algebraic tensor product.  
	
	Thus any essential $*$-algebra $\mathcal{A}$ is an essential bornological algebra with the fine bornology.  The bornological multiplier algebra is just the algebraic multiplier algebra.
\end{ex}

Let $\Delta : \mathcal{A} \rightarrow M(\mathcal{A}\hat{\otimes} \mathcal{A})$ be a $*$-homomorphism, The maps $\mathcal{A}\hat{\otimes} \mathcal{A}\rightarrow M(\mathcal{A}\hat{\otimes} \mathcal{A})$ 
$$\gamma_l : f\otimes g\mapsto \Delta(f)(g\otimes 1),~~~~~\gamma_r : f\otimes g\mapsto \Delta(f)(1\otimes g),$$
are called left Galois maps associated to $\Delta$ and 
$$\rho_l : f\otimes g\mapsto (f\otimes 1)\Delta(g),~~~~~\rho_r : f\otimes g\mapsto (1\otimes f)\Delta(g) $$
the right Galois maps.\\
\begin{rmk}
Since in our case we consider $*$-algebras, note that right Galois maps can be recovered from left ones by composing with the involution $*$.
\end{rmk}
If we suppose that $\Delta : \mathcal{A} \rightarrow M(\mathcal{A}\hat{\otimes} \mathcal{A})$ is essential, then one can define $(\Delta\hat{\otimes} \text{id})\circ \Delta$ and $( \text{id}\hat{\otimes} \Delta)\circ \Delta$ as maps from $ \mathcal{A} $ to $ M(\mathcal{A}\hat{\otimes} \mathcal{A}\hat{\otimes} \mathcal{A})$. If these maps coincide then we say that the homomorphism $\Delta$ is coassociative.

\begin{defin}
Definition 4.1. An essential $*$-homomorphism $\Delta : \mathcal{A} \rightarrow M(\mathcal{A}\hat{\otimes} \mathcal{A})$ is called  a comultiplication if is coassociative. 
\end{defin}
\begin{defin}
Let $\Delta : \mathcal{A} \rightarrow M(\mathcal{A}\hat{\otimes} \mathcal{A})$ be a comultiplication such that all Galois
maps associated to $\Delta$ define bounded linear maps from $\mathcal{A}\otimes \mathcal{A}$ into itself.
A bounded linear functional $\phi : \mathcal{A}\rightarrow \mathbb{C} $ is called left invariant if for all $a\in \mathcal{A}$, $(\iota\hat{\otimes}\phi)(\Delta(a))=\phi(a)1 $. Similarly, a bounded linear functional  $\phi : \mathcal{A}\rightarrow \mathbb{C} $ is called right invariant if for all $a\in \mathcal{A}$, $(\phi\hat{\otimes} \iota)(\Delta(a))=\phi(a)1. $

\end{defin}

\begin{defin}
A bornological quantum group is an essential bornological $*$-algebra  $\mathcal{A}(\mathbb{G})$ satisfying the approximation property, together with a $*$-preserving comultiplication $\Delta : \mathcal{A}(\mathbb{G})\rightarrow M(\mathcal{A}(\mathbb{G})\hat{\otimes} \mathcal{A}(\mathbb{G}))$, such that all Galois maps associated to $\Delta$ are isomorphisms, and a faithful left invariant positive functional $\phi_{\mathbb{G}}$.\end{defin}

According to \cite[Theorem 4.8]{Born}, the hypothesis on Galois maps ensures that there exists a uniquely determined bounded homomorphism $\epsilon : \mathcal{A}(\mathbb{G})\rightarrow \mathbb{C}$ and a linear isomorphism $S : \mathcal{A}(\mathbb{G})\rightarrow \mathcal{A}(\mathbb{G})$ which is both an algebra and coalgebra antihomomorphism such that 
$$(\epsilon\hat{\otimes} \iota)\circ\Delta=\iota=(\iota\hat{\otimes} \epsilon)\circ\Delta $$
and
$$\mu(S\hat{\otimes}\iota)\circ \gamma_r=\epsilon\hat{\otimes} \iota~~~~~\text{and}~~~~~\mu(\iota\hat{\otimes} S)\circ \gamma_l=\iota\hat{\otimes }\epsilon.$$
where $\mu : \mathcal{A}(\mathbb{G})\hat{\otimes} \mathcal{A}(\mathbb{G})\rightarrow \mathcal{A}(\mathbb{G})$ designates the multiplication of $\mathcal{A}(\mathbb{G}).$
The functional $\phi_{\mathbb{G}}$ is called a left Haar integral.

We will often use Sweedler notation  $\Delta(f) = f_{(1)}\otimes f_{(2)}$ to denote the coproduct of $f\in\sA$.  For bornological quantum groups, this is a purely formal notation to designate the position of the multiplier $\Delta(a)$ in the legs of a tensor product expression.  

\begin{ex}
	If $G$ is a Lie group, then $\sA(G) = C^\infty_c(G)$ is a bornological quantum group when equipped with the bornology of its usual $LF$-topology, the pointwise product, the coproduct given by pullback along the group law $G\times G\rightarrow G$, and where $\phi_G$ is integration with respect to left Haar measure.	
\end{ex}

\begin{ex}
	An algebraic quantum group is a bornological quantum group with the fine bornology (where the bounded sets are the compact subsets of finite dimensional subspaces).
\end{ex}

\begin{prop}
There is a unique bounded algebra automorphism $\sigma:\sA\to\sA$ such that $\phi(ab) = \phi_{\mathbb{G}}(b\sigma(a))$ for all $a,b\in\sA$.
\end{prop}
See \cite[Proposition 5.3]{Born} for details on this automorphism. We also give two properties that will be used later in this paper.
\begin{prop}
We have $\Delta \circ \sigma = (S^2 \otimes \sigma)\circ\Delta$ and $\sigma(\overline{a}) = \overline{\sigma^{-1}(a)}$ for all $a\in\sA$.
\label{sigma}
\end{prop}

\paragraph{\textbf{The modular element}}
\begin{prop}
 There exists an invertible self-adjoint element $\delta_{\mathbb{G}}\in M(\mathcal{A}(\mathbb{G}))$, called the modular element associated with the Haar state $\phi_{\mathbb{G}}$, defined by the property 
$$(\phi_{\mathbb{G}}\hat{\otimes} \iota)(\Delta(f))=\phi_{\mathbb{G}}(f)\delta_{\mathbb{G}}\in M(\mathcal{A}(\mathbb{G})),~\forall f\in \mathcal{A}(\mathbb{G}).$$
We also mention the notable property
$$\phi_{\mathbb{G}}(S(f))=\phi_{\mathbb{G}}(f\delta_{\mathbb{G}}).$$
\end{prop}
\begin{thm}
 \label{thm:delta-z}
 For all $z\in \mathbb{C}$, there exists a unique bounded mutliplier of $\mathcal{A}(\mathbb{G})$ denoted $\delta^z_{\mathbb{G}}$ such that
 \begin{enumerate}
     \item For any $z\in \mathbb{C}$, $\overline{\delta^z_{\mathbb{G}}}=\delta^{\bar{z}}_{\mathbb{G}}$
     \item For any $y,z\in \mathbb{C}$, $\delta^y_{\mathbb{G}}\delta^z_{\mathbb{G}}=\delta^{y+z}_{\mathbb{G}}$,
     \item For any $t\in\mathbb{R}, \delta^{it}_{\mathbb{G}}$ is unitary in ${M}(\mathcal{A}(\mathbb{G}))$,
     \item For any $t\in\mathbb{R}$, $\delta^{t}_{\mathbb{G}}$ is a positive element, in the sense that $\delta^{t}_{\mathbb{G}}=\delta^{t/2}_{\mathbb{G}}\delta^{t/2}_{\mathbb{G}}$ and $\delta^{t/2}_{\mathbb{G}}$ is a self adjoint element.
     \end{enumerate}
\end{thm}
This result is Theorem 3.27 in \cite{RY}. In the present paper we will only use the element $ \delta^{1/2}_{\mathbb{G}}$.
\begin{rmk}
We recall that in \cite{RY}, we have made the hypothesis that $\sigma_\GG(\delta_\GG)=\delta_\GG$, that is the scaling constant equals $1$. This assumption is also made in this paper.
\end{rmk}
\begin{prop}\label{prop:sbar}
There exists an automorphism $|S|$ of $\sA(\GG)$ such that $|S|^2=S^2$, $\phi_\GG\circ|S|=\phi_\GG$ and $|S|(\bar{f})=\overline{|S|^{-1}(f)}$ for all $f\in\sA(\GG)$.
\end{prop}
\paragraph{\textbf{Pontryagin Duality}}

We Recall that we define the space  $\mathcal{A}(\mathbb{\hat{G}})$ as a subspace of bounded linear functionals on $\mathcal{A}(\hat{\GG})$:
$$\mathcal{A}(\mathbb{\hat{\GG}})=\{\phi_{\mathbb{G}}(\cdot f),f\in \mathcal{A}(\hat{G})\}. $$
\begin{prop}
 The bornological space $\mathcal{A}(\mathbb{\hat{G}})$ is a bornologocal quantum group when it is equipped with the multiplier Hopf structure defined by duality :  Let
  $ x,y\in \mathcal{A}(\mathbb{\hat{G}}), f,g\in \mathcal{A}(\mathbb{G}),$
we have
 \begin{align*}
     (xy,f)&=(x\otimes y,\Delta(f)),\\
     (\hat{\Delta}(x),f\otimes g)&=(x,gf),\\
     \hat{\epsilon}(x)&=(x,1),\\
     (\hat{S}(x),f)&=(x,S^{-1}(f)),\\
     x^*(f)&=\overline{x(S(f)^*)}.
 \end{align*}
 Moreover it admits a left invariant integral defined by $\phi_{\hat{\mathbb{G}}}(\mathcal{F}(f))=\epsilon(f).$
\end{prop}

\justify The following result is Theorem 2.7 from \cite[Section 7]{Born}.
\begin{thm}
The double dual quantum group of $\mathcal{A}(\mathbb{G})$ is canonically isomorphic to $\mathcal{A}(\mathbb{G})$.
\end{thm}
In order to do calculations similar to the classical case when one considers the convolution algebra of a locally compact group $G$, we introduce the following notations.
\begin{defin}
We consider the $*$-bornological algebra $\mathcal{D}(\mathbb{G})$ with $\mathcal{D}(\mathbb{G})=\mathcal{A}(\mathbb{G})$ as a bornological vector space, and with product and involution given by
\begin{align*}
    f*g&=(\text{id}\otimes\phi_{\mathbb{G}})[(1\otimes S^{-1}(g))(\Delta(f))],\  \forall f,g\in \mathcal{D}(\mathbb{G}),\\
    f^*&=\overline{S(f)}\delta_{\mathbb{G}}.
\end{align*}\end{defin}

\justify In what follows we will use $f\mapsto \Bar{f}$ to denote the $*$-involution of $\mathcal{A}(\mathbb{G})$ to avoid confusion with the $*$-involution of $\mathcal{D}(\mathbb{G})$.\\
\begin{prop}
The map $\mathcal{F} : \mathcal{D}(\mathbb{G})\rightarrow \mathcal{A}(\mathbb{\hat{G}})$ is an isomorphism of $*$-bornological algebras. 
\end{prop}
\begin{rmk}
The reader should be careful that in whole paper we juggle with both $\mathcal{A}(\mathbb{G})$ and $\mathcal{D}(\mathbb{G})$ using everywhere both algebra structures, which can be confusing. 
\end{rmk}
\begin{prop}
  \label{lem:dual_ip}
  For any $f,g\in\sA(\mathbb{G})$ we have $\epsilon(f^* *g) = \phi(\overline{f}g)$.
\end{prop}
\begin{lm}
	\label{lem:convolution_coproduct_compatibility}
	For any $f,g \in \sA(\GG)$ we have the formal equalities
	\[
	 \Delta(f*g) = f_{(1)} \otimes (f_{(2)}*g) = (f*g_{(1)})\otimes g_{(2)}.
	\]
	More precisely, for any $a\in\sA(\GG)$ we have
	\begin{align*}
	 (a\otimes 1)\Delta(f*g) &= af_{(1)} \otimes (f_{(2)}*g) &
	 \Delta(f*g)(a\otimes 1) &= f_{(1)}a \otimes (f_{(2)}*g) \\
   (1\otimes a)\Delta(f*g) &= (f*g_{(1)}) \otimes a g_{(2)} &
	 \Delta(f*g)(1\otimes a) &= (f*g_{(1)}) \otimes g_{(2)}a,
	\end{align*}
	where the right hand side of the first equation is understood by first applying a Galois map to $a\otimes f$ and then taking the convolution with $g$ in the second leg, and similarly for the others.
\end{lm}
\paragraph{\textbf{Modules over a bornological quantum group.}}If $\mathcal{A}(\mathbb{G})$ is a unital bornological quantum group, then an essential left corepresentation of $\mathcal{A}(\mathbb{G})$ (also called a left coaction of $\mathbb{G}$ or an essential left $\mathcal{A}(\mathbb{G})$-comodule) is a bounded linear map
\[
 \alpha: V \to \mathcal{A}(\mathbb{G}) \otimes V
\]
which satisfies the coassociativity and essentiality conditions
\begin{align*}
 (\id\otimes\alpha)\alpha &=  (\Delta\otimes\id)\alpha ,\\
 (\epsilon\otimes\id)\alpha &= \id.
\end{align*}
If $\mathcal{A}(\mathbb{G})$ is not unital, this definition needs to be adjusted.  A corepresentation is then defined as a linear map
\[
 \alpha : V \to \Hom_{\mathcal{A}(\mathbb{G})}(\mathcal{A}(\mathbb{G}), \mathcal{A}(\mathbb{G})\otimes V),
\]
where $\mathcal{A}(\mathbb{G})$ and $\mathcal{A}(\mathbb{G})\otimes V$ are given the natural left $\mathcal{A}(\mathbb{G})$-actions.  The required coassociativity relation on $\alpha$ is given as a pentagonal equation as follows.  Using the Hom-tensor adjunction, we can view $\alpha$ as an element of $\Hom(\mathcal{A}(\mathbb{G})\otimes V; \mathcal{A}(\mathbb{G})\otimes V)$.  Then we require
\[
 \alpha_{23}\, \alpha_{13}\, (\rho_l)_{12} = (\rho_l)_{12} \, \alpha_{23},
\]
where $\rho_l$ is the Galois map from above, and we are using the standard leg-numbering notation for maps on $\mathcal{A}(\mathbb{G}) \otimes \mathcal{A}(\mathbb{G}) \otimes V$.  Essentiality is the requirement that $\alpha$ define a linear isomorphism from $\mathcal{A}(\mathbb{G}) \otimes V$ to itself.

\paragraph{\textbf{Associated locally compact quantum group.}} 
Before moving to the next section, we briefly summarize results of \cite{RY} that establish the link between a bornological quantum group and its associated $C^*$-algebraic quantum group.

There exists a Hilbert space $L^2(\mathbb{G})$ and a $C^*$-algebra $C_0^r(\mathbb{G})\subset B(L^2(\mathbb{G}))$ together with a linear map $\Lambda : \mathcal{A}(\mathbb{G})\rightarrow L^2(\mathbb{G})$ and a bounded algebra homomorphism $m : \mathcal{A}(\mathbb{G})\rightarrow C_0^r(\mathbb{G})$ both with dense images such that $ f,g\in  \mathcal{A}(\mathbb{G})$:
\begin{itemize}
    \item $\prodscal{\Lambda(f)}{\Lambda(g)}=\phi_{\mathbb{G}}(\bar{f}g)$,
    \item $m(f)\Lambda(g)=\Lambda(fg)$.
\end{itemize}
Furthermore $C_0^r(\mathbb{G})$ is equipped with a comultiplication $\Delta$ and a left Haar weight which extend the comultiplication and Haar integral of $\sA(\GG)$.

In \cite{RY}, it is shown that this defines a locally compact quantum group in the sense of Kustermans and Vaes \cite{K&V}.  That article requires an additional technical hypothesis, namely that $\sA(\GG)$ admits an approximate unit $(e_n)$ such that both $(e_n)$ and $\sigma_{i/2}(e_n)$ converge to $1$ in the bornology of the multiplier algebra, and we shall impose this assumption here as well.  We do not know if this hypothesis is necessary, although it is easily checked in the natural examples, including all examples to be discussed here.  If ultimately, as we expect, this condition is shown to be unnecessary for obtaining the locally compact quantum group $C^u_r(\GG)$, then it can be removed from the present article as well.  All we require is that the $C^*$-completion $C^0_r(\GG)$ of $\sA(\GG)$ is a locally compact quantum group.

The dual reduced quantum group is denoted by $C^*_r(\mathbb{G})$, and $C_0^u(\mathbb{G})$ and $C^*_u(\mathbb{G})$ refer to the associated universal locally compact quantum groups. Further, $L^\infty(\mathbb{G})$ and $\mathcal{L}(\mathbb{G})$ refer to the associated von Neumann algebraic quantum groups.\\
%We also have the following result:
%\begin{prop}Let $H$ be a Hilbert space. A bounded corepresentation of $ \mathcal{A}(\mathbb{G})$ on $H$ (where $H$ is endowed with its metric bornology) gives rise to a uniquely determined corepresensation of $(C_0^r(\mathbb{G}),\Delta)$ on $H$.

\section{Closed quantum subgroups}
\label{closed}
\begin{defin}
A bornological quantum group $ \mathcal{A}(\mathbb{B})$, equipped with a bounded surjective $*$-morphism of bornological quantum groups $\pi : \mathcal{A}(\mathbb{G}) \rightarrow \mathcal{A}(\mathbb{B})$ is called a closed quantum subgroup of $ \mathcal{A}(\mathbb{G})$.
\end{defin}

Let $\mathcal{A}(\mathbb{B})$ be such a quantum subgroup with a left Haar integral $\phi_{\mathbb{B}}$. 

\begin{rmk}\label{rmk:hyp}
In general we have $\sigma_\BB(\pi(\delta_{\mathbb{G}}))=\mu\pi(\delta_{\mathbb{G}})$ for some complex number $\mu$ with modulus $1$. As for the scaling constant, we make the hypothesis here that this constant equals $1$. We thus have $\sigma_\BB(\pi(\delta_{\mathbb{G}}^{\frac{1}{2}}))=\pi(\delta_{\mathbb{G}}^{\frac{1}{2}}).$
\end{rmk}

The convolution algebras $\sD(\GG)$ and $\sD(\BB)$ are, by definition, identified as linear spaces with the spaces $\sA(\GG)$ and $\sA(\BB)$. Therefore the map $\pi : \sA(\GG)\to\sA(\BB)$ can also be seen as a map from $\sD(\GG)$ to $\sD(\BB)$. However, as it stands, this map does not have the properties of what we will call a \textit{generalized conditional expectation}. Instead, we first define
$$\gamma=\pi(\delta_{\mathbb{G}}^{-\frac{1}{2}})\delta_{\mathbb{B}}^{\frac{1}{2}} \in M(\mathcal{A}(\mathbb{B})),$$
which is a group-like element.
Then we modify the map $\pi$ to 
$$E:\mathcal{D}(\mathbb{G})\rightarrow \mathcal{D}(\mathbb{B}),~E(f)=\pi(f)\gamma.$$
In order to describe the relevant properties of $E$, we must start with some preliminaries concerning the action of $\sD(\BB)$ on $\sD(\GG)$.
We consider the dual morphism $\hat{\pi} : \mathcal{D}(\mathbb{B})\rightarrow M(\mathcal{D}(\mathbb{G}))$ and set for all $f$ in $\mathcal{D}(\mathbb{G})$ and for all $h\in \mathcal{D}(\mathbb{B})$,
\begin{align*}
    f\cdot h=f*\hat\pi(h\gamma).
\end{align*}
\begin{prop}\label{prop:action}
The map $f\mapsto f\cdot h$ defines a right action of the algebra $\mathcal{D}(\mathbb{B})$ on the space $\mathcal{D}(\GG)$.
\end{prop}
\begin{proof}
    Let $h,k\in \mathcal{D}(\BB)$. Since $\gamma$ is group-like we have $(h*k)\gamma=h\gamma*k\gamma$ and thus for $f\in\mathcal{D}(\GG)$ we have
    \begin{align*}
        f\cdot(h*k)&=f*\hat\pi(h\gamma)*\hat\pi(k\gamma)\\
        &=(f\cdot h)\cdot k.
    \end{align*}
\end{proof}
We are going to prove that $E$ preserves the *-involution and has a ``conditional expectation'' property with respect to this action.
\begin{lm}
The two multipliers $\delta_\BB$ and $\pi(\delta_\GG)$ commute.
\end{lm}
\begin{proof}We know that we have 
$$\phi_\BB(S(h))=\phi_\BB(h\delta_\BB),$$
for all $h\in \sA(\BB)$. By our hypothesis in Remark \ref{rmk:hyp} we also have that $\sigma_\BB(\pi(\delta_\GG^{-1}))=\pi(\delta_\GG^{-1})$.
    Let then $h\in\sA(\BB)$. We have $\phi_\BB(S(\pi(\delta_\GG)h))=\phi_\BB(S(h\pi(\delta_\GG)))$. On the one hand this gives
    \begin{align*}
        \phi_\BB(S(\pi(\delta_\GG)h))&=\phi_\BB(\pi(\delta_\GG)h\delta_\BB)\\
        &=\phi_\BB(h\delta_\BB\pi(\delta_\GG)),
    \end{align*}
    and on the other
    $$\phi_\BB(S(h\pi(\delta_\GG)))=\phi_\BB(h\pi(\delta_\GG)\delta_\BB). $$
    Therefore $\delta_\BB\pi(\delta_\GG)=\pi(\delta_\GG)\delta_\BB$.
\end{proof}
\begin{lm}
\label{lem:dual_to_restriction}
Let $h\in\sD(\HH)$.  The convolution multiplier $\hat{\pi}(h) \in M(\sD(\GG))$ is given by
 \begin{align*}
	  \hat{\pi}(h)*f &= \phi_{\mathbb{H}}(S^{-1}(\pi(f_{(1)}))h) f_{(2)}, \\
	  f*\hat{\pi}(h) &= f_{(1)}\, \phi_{\mathbb{H}}(\pi(\delta_{\mathbb{G}} S(f_{(2)}))h)
	    = f_{(1)}\, \phi_{\mathbb{H}} (S^{-1}(h)\pi(f_{(2)})\pi(\delta_{\mathbb{G}}^{-1})\delta_{\mathbb{H}}).
 \end{align*}
 for all $f\in\sD(\GG)$
\end{lm}
 \begin{proof}
 Let $f\in\mathcal{D}(\mathbb{G})$, $h\in \mathcal{D}(\mathbb{B})$ and $a\in \sA(\GG)$. Using the duality between $\sA(\GG)$ and $\sD(\GG)$ and taking into account the left invariance of $\phi_\GG$ and the definition of $\hat\pi$ we get
 \begin{align*}
     (\hat\pi(h)*f,a)&=(\hat\pi(h),a_{(1)})(f,a_{(2)})\\
     &=(h,\pi(a_{(1)}))(f,a_{(2)})\\
     &=\phi_\BB(\pi(a_{(1)})h)\phi_\GG(a_{(2)}f)\\
     &=\phi_\BB(\pi(S^{-1}(f_{(1)}))h)\phi_\GG(af_{(2)})\\
     &=(\phi_\BB(\pi(S^{-1}(f_{(1)}))h)f_{(2)},a).
 \end{align*}
Similarly, using this time the right relative invariance we get 
\begin{align*}
    (f*\hat\pi(h),a)&=(f,a_{(1)})(\hat\pi(h),a_{(2)})\\
     &=\phi_\GG(a_{(1)}f)\phi_\BB(\pi(a_{(2)})h)\\
     &=\phi_\GG(af_{(1)})\phi_\BB(\pi(\delta_\GG S(f_{(2)}))h)\\
     &=\phi_\GG(af_{(1)})\phi_\BB(S(S^{-1}(h)\pi(f_{(2)}\delta_\GG^{-1})))\\
     &=\phi_\GG(af_{(1)})\phi_\BB(S^{-1}(h)\pi( f_{(2)}\delta_\GG^{-1})\delta_\BB)\\
     &=(\phi_\BB(S^{-1}(h)\pi( f_{(2)})\pi(\delta_\GG^{-1})\delta_\BB)f_{(1)},a).
\end{align*}
 \end{proof}
 \begin{rmk}
 Note that since $\pi(\delta_\GG)$ and $\delta_\BB$ commute, we have $\pi(\delta_\GG^{-1})\delta_\BB=\gamma^2$.
 \end{rmk}

\begin{prop}\label{cond}
The map $E:\mathcal{D}(\mathbb{G})\rightarrow \mathcal{D}(\mathbb{B})$, $E(f)=\pi(f)\gamma$, has the two following properties :
\begin{enumerate}
    \item $E(f^*)=E(f)^*$, for all $f\in \mathcal{D}(\mathbb{G})$,
    \item $E(f\cdot h)=E(f)*h$. for all $f\in \mathcal{D}(\mathbb{G})$ and $h\in \mathcal{D}(\mathbb{B})$.
\end{enumerate}
The map $E$ is the \textit{generalized conditional expectation} we were looking to build.
\end{prop}

\begin{proof}
Let $f\in \mathcal{D}(\mathbb{G})$. We have
\begin{align*}
    E(f^*)&=E(\overline{S(f)}\delta_{\mathbb{G}})\\
    &=\overline{S(\pi(f))}\pi(\delta_{\mathbb{G}})\gamma\\
    &=\overline{S(\pi(f))}\pi(\delta_{\mathbb{G}}^{\frac{1}{2}})\delta_{\mathbb{B}}^{\frac{1}{2}}\\
    &=\overline{S(\pi(f)\gamma)}\delta_{\mathbb{B}}=E(f)^*.
    \\
    \end{align*}
    Now let $h\in \mathcal{D}(\mathbb{B})$. Using that $\sigma(\gamma^{-1})=\gamma^{-1}$ we get
    \begin{align*}
        E(f\cdot h)&=E(f*\hat\pi(h\gamma))\\
        &=(\id\hat\otimes\phi_\BB)((1\otimes \gamma^{-1}S^{-1}(h))(\pi\hat\otimes\pi)(\Delta(f))(1\otimes\gamma^2))\gamma\\
        &=(\id\hat\otimes\phi_\BB)((1\otimes S^{-1}(h))(\pi\hat\otimes\pi)(\Delta(f))(\gamma\otimes\gamma))\\
        &=E(f)*h.
    \end{align*}
\end{proof}

\section{The induction module $\mathcal{E}(\mathbb{G})$}
We now make the assumption that the quantum subgroup $\BB$ is amenable, that is, $C^*_u(\BB)=C^*_r(\BB)$. The goal of this section is to define a Hilbert $C^*(\BB)$-module with a left $C^*_u(\GG)$-action by completing $\sD(\GG)$. We equip the space $\mathcal{D}(\mathbb{G})$ with the right action of $\mathcal{D}(\mathbb{B})$ defined as in Proposition \ref{prop:action}.
\begin{defin}
Let $\sV$ be a right $\sD(\BB)$-module. A $\sD(\BB)$-valued inner product on $\sV$ will mean a sesquilinear map $\prodscal{\cdot}{\cdot} : \sV\times\sV\to \sD(\BB)$ such that for all $v,w\in \sV$ and $h\in\mathcal{D}(\mathbb{B})$ we have
\begin{enumerate}
    \item $\prodscal{v}{w\cdot h}=\prodscal{v}{w}*h$,
    \item $\prodscal{v}{w}^*=\prodscal{w}{v}$.
    \item $\lambda_\BB(\prodscal{v}{v})$ is a positive element of $C^*(\BB)$ and $\prodscal{v}{v}=0 \Leftrightarrow v=0.$
\end{enumerate}
Such a module $\sV$ endowed with a $\sD(\BB)$-valued inner product will be called a \emph{$\sD(\BB)$-inner product space}.
\end{defin}
\begin{rmk}
The fact that we have to call on the regular representation $\lambda_\BB$ is not very aesthetic but this is because the notion of positivity in the bornological quantum group $\sD(\BB)$ cannot be defined intrinsically. Below, we will prove positivity in the reduced $C^*$-algebra, but in fact we would want to prove positivity in the universal $C^*$-algebra. In the classical case, Rieffel \cite{Rieffel} uses the existence of a Bruhat section to prove positivity of the scalar product in the universal $C^*$-algebra. Because we do not have a suitable analogue of this in the quantum world we will content ourselves with the case of amenable quantum subgroups. Nonetheless we expect that the construction could be extended to any quantum subgroups.
%Nevertheless, assuming that a Bruhat exists as an element of the bornological algebra $\sA(\GG)$ does not seem reasonable. One thus should have to suppose only that this Bruhat section belongs to $C_0^u(\GG)$ and we do not know how to work with such object in this context.
\end{rmk}

\begin{prop}\label{prodscal} The sesquilinear map $\prodscal{~}{~ }_{\mathcal{D}(\mathbb{B})}$ defined for $f,\ g\in \mathcal{D}(\mathbb{G})$ by 
\[\prodscal{f}{g}_{\mathcal{D}(\mathbb{B})}=E(f^**g),\]
\\
defines a $\mathcal{D}(\mathbb{B})$-valued inner product.
\end{prop}
To prove the previous proposition we note first that the $\mathcal{D}(\mathbb{B})$-linearity and compatibility with the involution of the above sequilinear map follow immediately from Proposition \ref{cond}.
It only remains to check the strict positivity, which will be a consequence of Proposition \ref{map} below.
\begin{rmk}\label{rmk:prodscalexpress}
Let $f,\ g\in \mathcal{D}(\mathbb{G})$. We have
\begin{align*}
  E(f^**g)&=\phi_\GG(S^{-1}(g_{(1)})\overline{S(f)}\delta_\GG)\pi(g_{(2)})\gamma\\
  &=\phi_\GG(\bar{f}g_{(1)})\pi(g_{(2)})\gamma.
\end{align*}

\end{rmk}
\begin{rmk}
In what follows we will often use the maps $\Lambda_\GG : \sA(\GG)\to L^2(\GG)$, $\lambda_\GG : \sD(\GG)\to B(L^2(\GG))$ and the analogous maps $\Lambda_\BB$ and $\lambda_\BB$, but we will only write $\Lambda$ and $\lambda$. Their relation to $\BB$ or $\GG$ will depend on the context.
\end{rmk}

\begin{prop}\label{map}
The linear map $\rho_{\bullet} : \mathcal{D}(\mathbb{G})\rightarrow B(L^2(\mathbb{B}),L^2(\mathbb{G}))$ defined by $f\mapsto \rho_f$ where
$$\rho_f(\Lambda(\eta))=\Lambda(f\cdot\eta),\  \forall f\in \mathcal{D}(\mathbb{G}), \forall \eta\in \mathcal{D}(\mathbb{B}),$$

satisfies
$$\lambda_\BB(\prodscal{f}{g}_{\mathcal{D}(\mathbb{B})})=\rho_f^*\rho_g.$$
\end{prop}
\begin{proof}
First we claim that, as an operator from $L^2(\mathbb{G})$ to $L^2(\mathbb{B})$, $(\rho_f)^*$ acts on elements of $\Lambda(\sA(\GG))$ as
$$(\rho_f)^*~:~\Lambda(g)\mapsto\Lambda(\prodscal{f}{g}_{\mathcal{D}(\mathbb{B})}). $$
For this, note that using $\epsilon_{\mathbb{B}}(E(x))=\epsilon_{\mathbb{G}}(x)$, for any $x \in \mathcal{D}(\mathbb{G})$, we obtain $\epsilon_{\mathbb{B}}(E(x*y))=\epsilon_{\mathbb{G}}(x^**y)=\prodscal{x}{y}_{L^2(\mathbb{G})}$ for any $x,y\in \mathcal{D}(\mathbb{G})$. Therefore, for all $\eta\in \mathcal{A}(\mathbb{B})$ and $\xi\in  \mathcal{A}(\mathbb{G})$ we have
\begin{align*}
    \prodscal{\rho_f\Lambda(\eta)}{\Lambda(\xi)}_{L^2(\mathbb{G})}&=\prodscal{\Lambda(f\cdot\eta)}{\Lambda(\xi)}_{L^2(\mathbb{G})}\\
    &=\epsilon_{\mathbb{B}}(\prodscal{f\cdot\eta}{\xi}_{\mathcal{D}(\mathbb{B})})\\
    &=\epsilon_{\mathbb{B}}(\eta^**\prodscal{f}{\xi}_{\mathcal{D}(\mathbb{B})})\\
    &=\prodscal{\Lambda(\eta)}{\Lambda(\prodscal{f}{\xi})}_{L^2(\mathbb{B})}.
\end{align*}
We therefore have : 

\begin{align*}
  \Lambda( \prodscal{f}{g}_{\mathcal{D}(\mathbb{B})}*\eta) &= \Lambda(\prodscal{f}{g\cdot\eta}_{\mathcal{D}(\mathbb{B})})\\
   &=\rho_f^*\rho_g\Lambda(\eta) .
\end{align*}
\end{proof}
This concludes the proof of Proposition \ref{prodscal}. We also record the explicit formula
$$ \prodscal{f}{g}_{\mathcal{D}(\mathbb{B})}=(\phi_{\mathbb{G}}\hat\otimes\id)((\bar{f}\otimes 1)\Delta(g)(1\otimes\gamma)).$$
\begin{defin}
The Hilbert  $C^*(\BB)$-module obtained by completing $\mathcal{D}(\mathbb{G})$ with respect to the inner product above is denoted $\mathcal{E}(\mathbb{G})$ and we call it the induction module (associated to $\mathbb{B}$).
\end{defin}
See \cite{lance} for details about the completion. The space $\mathcal{E}(\mathbb{G})$ is innately equipped with a left $C^*_u(\mathbb{G})$-action, which commutes with the right $C^*(\BB)$-action. We then get our induction  bi-module
$$_{C^*_u(\mathbb{G})}\mathcal{E}(\mathbb{G})_{C^*(\BB)}.$$
Now, for $\alpha$ a representation of $C^*(\BB)$ on an $A$-Hilbert module  $K$ (where $A$ is any $C^*$-algebra) we consider, following Rieffel's definition for induced representations in \cite{Rieffel}, the $A$-Hilbert module 
$$\text{Ind}_{\mathbb{B}}^{\mathbb{G}}V=\mathcal{E}(\mathbb{G})\hat\otimes_{C^*(\BB)}V,$$
where the tensor product is completed with respect to the interior inner product \cite[Proposition 4.5]{lance}.

\section{Link with Vaes' approach to induction}

We consider in this Section our bornological quantum groups $\GG$ and $\BB$ as locally compact quantum groups, as described at the end of Section \ref{sec:BQG}, and we {assume} that $\BB$ is a closed quantum subgroup of $\GG$ in the sens of the following definition.
\begin{defin}
	\label{def:Vaes_subgroup}
	Let $\pi : C_0^u(\mathbb{G})\rightarrow M(C_0^u(\mathbb{B}))$ be an homomorphism. We say that $\pi$ identifies $\mathbb{B}$ as a \emph{closed quantum subgroup of $\mathbb{G}$ in the sense of Vaes} if there exists a faithful, normal, unital $*$-homomorphism $\hat\pi :  \mathcal{L}(\mathbb{B})\rightarrow \mathcal{L}(\mathbb{G})$ such that the following diagram commutes :
	\[
	 \xymatrix{
	  C^*_u(\mathbb{B}) \ar[r]^{\hat\pi} \ar[d]_{\lambda_{\mathbb{B}}} &
	   M(C^*_u(\mathbb{G})) \ar[d]^{\lambda_{\mathbb{G}}} \\
	  \mathcal{L}(\mathbb{B}) \ar[r]^{\hat\pi} &
	   \mathcal{L}(\mathbb{G})
   }
	\]
	where the vertical maps are the regular representations.
\end{defin}

\begin{rmk}
This definition is the notion used by Vaes in this work on induction for locally compact quantum groups \cite{Vaes}. There is another potentially weaker definition of closed quantum group that can be found in \cite{Daws}.
\end{rmk}

%Let us recall that we have at the bornological level a map $\hat{\pi} : \sD(B)\to M(\sD(\GG))$ defined in Section \ref{sec:born_subgroup}, and its operator algebraic version $\hat{\pi} : \sL(B)\to \sL(\GG)$, encountered in Theorem \ref{thm:subgroup}. We recall that those two maps are linked via the relation  $\lambda(\hat{\pi}(f))=\hat{\pi}(\lambda(f)),$ for all $f\in \sD(\GG)$. Throughout this section we will be using the definitions and notations of Section \ref{sec:lcqg}.

Let us summarize the induction procedure of \cite{Vaes} for locally compact quantum groups. We begin with some definitions and results from \cite[Section 3]{Vaes}. We consider $(A,\Delta)$ a locally compact quantum group with von Neumann algebra $M$ and GNS Hilbert space $H$. We also fix a $C^*$-algebra $B$. If $\sV$ is a $C^*$-B-module we write $\sL(\sV)$ for the $*$-algebra of adjointable $B$-linear operators.
\begin{defin}
Let $N$ be a von Neumann algebra and $\sV$ a $C^*$-$B$-module. A unital *-homomorphism $\beta : N\to\sL(\sV) $ is said to be strict (or normal) if it is strong* continuous on the unit ball of $N$.
\end{defin}
\begin{defin}
Let $M$ and $N$ be von Neumann algebras. We say that a $C^*$-$B$-module $\sV$ is a $B$-correspondence from $N$ to $M$ if we have
\begin{itemize}
    \item a strict *-homomorphism $\beta_l : M\to \sL(\sV)$,
    \item a strict *-antihomomorphism $\beta_r : N\to \sL(\sV)$,
    such that $\beta_l(M)$ and $\beta_r(N)$ commute.
\end{itemize}
\end{defin}
\begin{rmk}
In \cite{Vaes} the notation $\pi$ is used instead of $\beta$. Here we keep $\pi$ to designate the morphism from $\sA(\GG)$ to $\sA(\BB)$.
\end{rmk}
We will denote $x\cdot v=\beta_l(x)v$ and $v\cdot y=\beta_r(y)v$ for all $x\in M$, $y\in N$ and $v\in\sV$ and this correspondence will be denoted as $_M\boxed{\sV}\vphantom{ }_N$.
\begin{prop}\emph{(\cite[Proposition 3.4]{Vaes})}.\label{prop:regtocor}
 Let $X\in \sL(A\hat\otimes\sV)$ be a unitary corepresentation on a $C^*$-$B$-module $\sV$. There is a $B$-correspondence  $_{\hat{M}}\boxed{H\hat\otimes\sV}\vphantom{ }_{\hat{M}}$ given by 
\[ x\cdot v = X(x\otimes 1)X^*v~\text{and}~v\cdot y=(\hat{J}_\GG y^*\hat{J}_\GG\otimes1)v~\text{for}~x,y\in\hat{M},~v\in H\hat\otimes\sV.\]
\end{prop}
\begin{defin}\emph{(\cite[Definition 3.5]{Vaes})}.
Let $_{\hat{M}}\boxed{\mathcal{F}}\vphantom{ }_{\hat{M}}$ be a $B$-correspondence from $\hat{M}$ to $\hat{M}$ and suppose $\beta : M'\to \sL(\mathcal{F})$ is a strict *-homomorphism. We say that $\beta$ is bicovariant when
\begin{align*}
(\beta_l\hat\otimes\id)(\hat\Delta(x))&=(\beta\hat\otimes\id)(\hat{V})(\beta_l(x)\otimes1)(\beta\hat\otimes\id)(\hat{V}^*)~\text{ and }\\(\beta_r\hat\otimes\hat{R})(\hat\Delta(x))&=(\beta\hat\otimes\id)(\hat{V})(\beta_l(x)\otimes1)(\beta\hat\otimes\id)(\hat{V}^*),
\end{align*}
where $\hat{V}=(J\hat\otimes J)W(J\hat\otimes J)$ and $\hat{R}$ denotes the unitary antipode of $\hat{M}$, see \cite[Preliminaries]{Vaes}. In this case we call $\mathcal{F}$ a bicovariant $B$-correspondence and we write $\overset{M'}{_{\hat{M}}\boxed{\mathcal{F}} _{\hat{M}}}$.
\end{defin}
\begin{rmk}\label{rmk:bicov}
We give this definition because we will need to deal with bicovariant $B$-correspondences. However its technical aspect does not concern us directly. The core of this section is to show the equivalence between two different bicovariant $B$-correspondences, where their structure is already provided by the results of \cite{Vaes}. Showing such an equivalence is simply a matter of showing that the morphisms satisfy the right commutation relations.
\end{rmk}
According to \cite[Remark 3.6]{Vaes}, we have a structure of bicovariant $B$-correspondence $\overset{M'}{_{\hat{M}}\boxed{H\hat\otimes\mathcal{V}} _{\hat{M}}}$ where the $B$-correspondence is given by Proposition \ref{prop:regtocor} and $\beta : M'\to \sL(H\hat\otimes\mathcal{V})$ is given by $\beta(x)=x\otimes 1$.
\begin{rmk}
It should be noted that there is a slight difference in conventions between the current work and the article of Vaes. Namely the skew-pairing between $\sA(\hat{\GG})$ and $\sA(\GG)$ is such that the coproduct on $\sA(\hat{\GG})$ is reversed in our conventions, while it is the multiplication which is reversed in Vaes’ conventions. Given that the modules discussed here are defined primarily in terms of $\sD(\GG)$-actions, this means that the action of the function algebra $M’ = L^\infty(\GG)’$ in the bicovariant modules we define below will be intertwined by the unitary antipode $R$. This forces us to slightly modify the definition of the morphism $\beta$ so that $\beta(x)=R(Jx^*J)\otimes 1$. 

In practice, this means the following.  If $a\in\sA(\GG)$ then the action of $m’(a)\in M’$ on the GNS space $H=L^2(\GG)$ in our conventions needs to be defined as
\[
  m’(a)\cdot \Lambda(\xi) = \Lambda(R(a)\xi),
\]
where $\xi\in\sA(\GG)$ and $R$ designates the unitary antipode of $M$. The fact that $R$ stabilizes the bornological algebra $\mathcal{A}(\mathbb{G})$ is a consequence of \cite{RY}. 
\end{rmk}
The following proposition is crucial to Vaes' induction procedure.  It will be the key result that we use to establish the equivalence between our approach to induction and Vaes'.
\begin{prop}\emph{(\cite[Proposition 3.7]{Vaes})}\label{prop:vaes}
If $\overset{M'}{_{\hat{M}}\boxed{\mathcal{F}} _{\hat{M}}}$ is a bicovariant $B$-correspondence, there exists a canonically determined $C^*$-$B$-module $\sE$ and a corepresentation $X\in\sL(A\hat\otimes\sE)$, unique up to equivalence, such that 
$$\overset{M'}{_{\hat{M}}\boxed{\mathcal{F}} _{\hat{M}}}\cong \overset{M'}{_{\hat{M}}\boxed{H\hat\otimes\sE} _{\hat{M}}} $$
as bicovariant correspondences. So, we get a bijective relation between unitary corepresentations on $C^*$-$B$-module and bicovariant $B$-correspondences.
\end{prop}
Note that from the corepresentation $X\in \sL(A\hat\otimes\sV)$, we obtain a *-morphism $\alpha : \hat{A}^u\to\sL(\sV)$ which verifies
 $$(\id\hat\otimes\alpha)(W^u)=X, $$
 where $W^u$ designates the universal multiplicative unitary of the quantum group $(A,\Delta)$.
 
 We now set $A=C_0^r(\GG)$ and thus we have $H=L^2(\GG)$, $M=L^\infty(\GG)$ and $\hat{M}=\sL(\GG)$. Let $X\in \sL(C^r_0(\GG)\hat\otimes\sV)$ be a corepresentation of $\mathbb{G}$ on a Hilbert $B$-module $\sV$. We still denote by $\alpha$ the corresponding $*$-morphism $\alpha : C^*_u(\mathbb{G})\rightarrow \sL(\sV)$, as well as its bornological version, $\alpha : \sD(\BB)\to \sL(\sV)$, which can be defined by restriction of the original $\alpha$ to $\lambda^u(\sD(\GG))$.
 
 %After stating the definition of bicovariant $B$-correspondence  \cite[Definition 3.5]{Vaes}, the author states that there exists a structure of bicovariant $B$-correspondence $_{\hat{M}}\overset{M'}{\boxed{H\hat\otimes\sV}}\vphantom{ }_{\hat{M}}$ where the morphism $\beta : M'\to\sL(H\hat\otimes\sV)$ is given by $\beta(x)=x\otimes 1$. 

 \begin{rmk}\label{rmk:action}
 One can describe explicitly the structure of the bicovariant $B$-correspondence $ \overset{L^\infty(\GG)'}{_{\sL(\GG)}\boxed{L^2(\GG)\hat\otimes\sV}\vphantom{ }_{\sL(\GG)}}$. Let $f\in \mathcal{D}(\mathbb{G})$,~ $\xi\in\sA(\mathbb{G})$ and $v\in \sV$. We have
\begin{itemize}
    \item $\lambda(f)\cdot(\Lambda(\xi)\otimes v) =(\lambda\otimes \alpha)(\hat{\Delta}(f))(\Lambda(\xi)\otimes v)$,
    \item $(\Lambda(\xi)\otimes v)\cdot \lambda'(f)=\Lambda(\xi*f)\otimes v$,
    \item $\beta(m'(f))(\Lambda(\xi)\otimes v)=\Lambda(R(f)\xi)\otimes v$.
\end{itemize}
Let us remark that if our conventions were coherent with those of Vaes we would have a flipped coproduct $\hat{\Delta}^{\mathrm{op}}$ in first point. This is because in Proposition \ref{prop:regtocor}, the left action is defined by $x\cdot v = X(x\otimes 1)X^*v$, for $x\in\hat{M},~v\in H\hat\otimes\sV$ and we have $W(x\otimes 1)W^*=\hat{\Delta}^{\mathrm{op}}(x)$.
 \end{rmk}

%\begin{proof}For the first point, We consider $W^u\in C_0^r(\GG)\hat\otimes C^*_u(\GG)$ the universal multiplicative unitary associated to $\GG$. We have $$\Sigma W^u(\lambda(f)\otimes1)W^u^*\Sigma=(\lambda\hat\otimes\lambda^u)(\hat{\Delta}(f)),$$for $f\in \sD(\GG)$, where we recall that $\lambda^u$ designates the embedding $\lambda^u : \sD(\GG)\to C^*_u(\GG)$. it thus follows that $\Sigma W^u(\lambda(f)\otimes1)W^u^*\Sigma=(\lambda^u\hat\otimes\lambda)(\hat{\Delta}^{\mathrm{op}}(f))$ and finally\begin{align*}X(\lambda(f)\otimes1)X^*(\Lambda(\xi)\otimes v)&=((\alpha\circ \lambda^u)\otimes\lambda)(\hat{\Delta}^{op}(f))(\Lambda(\xi)\otimes v).\end{align*}The second point is immediate since $\hat{J}_\GG m'(f)^*\hat{J}_\GG\Lambda(\xi)=\Lambda(\xi f)$ for all $\xi\in\sA(\GG)$. For the last point, we just mentionned that $\beta(x)=x\otimes1$ for $x\in M'$. Thus $\beta(m'(f))(\Lambda(\xi)\otimes v)=(m'(f)\Lambda(\xi))\otimes v=\Lambda(\xi f))\otimes v$.\end{proof}
% We have that $Y$ is a corepresentation of $\hat{A}$ since $\hat{V}$ is itself a corepresentation of $\hat{A}$

%We now consider a corepresentation $X\in \sL(C^r_0(\BB)\hat\otimes\sV)$ of $\mathbb{B}$ on an Hilbert $B$-module $\sV$. We still denote by $\alpha$ the corresponding $*$-morphism $\alpha : C^*(\BB)\rightarrow \sL(\sV)$.

From now we consider $X\in \sL(C^r_0(\BB)\hat\otimes\sV)$ a corepresentation of $\mathbb{B}$ on a Hilbert $B$-module $\sV$, accompanied by the $*$-morphism $\alpha : C^*(\BB)\rightarrow \sL(\sV)$. The aim of the next paragraphs is to build the induced corepresentation of $\sV$ with Vaes' technique. Following \cite[Lemma 4.5]{Vaes} we consider the $B$-correspondence ${_{\sL(\BB)}\boxed{L^2(\GG)\hat\otimes\sV}\vphantom{ }_{\sL(\GG)}}$.
\begin{rmk}\label{rmk:action_on_the_false}
The morphisms in this structure of $B$-correspondence can be made explicit as is the previous remark. Let $f\in \mathcal{D}(\mathbb{G})$, $h\in\sD(\BB)$,~ $\xi\in\sA(\mathbb{G})$ and $v\in \sV$. We have
\begin{itemize}
    \item $\lambda(h)\cdot(\Lambda(\xi)\otimes v) =(\lambda\circ\hat{\pi}\otimes \alpha)(\hat{\Delta}(h))(\Lambda(\xi)\otimes v)$,
    \item $(\xi\otimes v)\cdot \lambda'(f)=(\xi*f)\otimes v$.
\end{itemize}
The second point does not differ from the formula in Remark \ref{rmk:action}. The first point requires justification. It is claimed in \cite[Lemma 4.5]{Vaes} that the morphism $\beta_l : \sL(\BB)\to \sL(L^2(\GG)\hat\otimes\sV)$ is characterized by the property 
$$\beta_l(a)(u\otimes1)\xi = (u \otimes1)X(a\otimes1)X^*\xi,$$
for every $a\in\sL(\BB)$, $\xi\in L^2(\BB)\hat\otimes\sV$ and $u\in B(L^2(\BB),L^2(\GG))$ satisfying $ux=\hat{\pi}(x)u$ for all $x\in \sL(\BB)$. Let then $u\in B(L^2(\BB),L^2(\GG))$ satisfying $ux=\hat{\pi}(x)u$ for all $x\in \sL(\BB)$ and let $h\in\sD(\BB)$, $\eta\in \sA(\BB)$ and $v\in \sV$. According to Remark \ref{rmk:action} we have
$$X(a\otimes1)X^*(\Lambda_\BB(\eta)\otimes v)=(\lambda_\BB\otimes \alpha)(\hat{\Delta}(h))(\Lambda_\BB(\xi)\otimes v).$$
Therefore, since $\lambda(h)\in \sL(\BB)$, we have
\begin{align*}
    [(u \otimes1)X(\lambda(h)\otimes1)X^*](\Lambda(\xi)\otimes v)&=(u \otimes1)((\lambda\otimes \alpha)(\hat{\Delta}(h)))(\Lambda(\xi)\otimes v)\\
    &=(\hat{\pi}\circ\lambda\otimes \alpha)(\hat{\Delta}(h))(u \otimes1)(\Lambda(\xi)\otimes v)\\
    &=(\lambda\circ\hat{\pi}\otimes \alpha)(\hat{\Delta}(h))(u \otimes1)(\Lambda(\xi)\otimes v)\\
    &=(\lambda\circ\hat{\pi}\otimes \alpha)(\hat{\Delta}(h))[(u \otimes1)(\Lambda(\xi)\otimes v)]
\end{align*}
and thus it coincides with what we claimed. Finally we note that we also have a *-morphism $\beta_{L^2(\mathbb{G})\hat\otimes \sV} : L^\infty(\GG)'\to \sL(L^2(\mathbb{G})\hat\otimes \sV)$ given by $\beta_{L^2(\mathbb{G})\hat\otimes \sV}(m'(f))(\Lambda(\xi)\otimes v)=\Lambda(R(f)\xi)\otimes v$.
\end{rmk}

We introduce the space $\mathcal{I}$ from \cite[Definition 4.2]{Vaes}:
$$\mathcal{I}=\{u \in B(L^2(\mathbb{B}),L^2(\mathbb{G})),\ ux=\hat{\pi}'(x)u\   \forall x\in \mathcal{L}(\mathbb{B})'\},$$
 where $\hat{\pi}'$ refers to the natural action of $\mathcal{L}(\mathbb{B})'$ on $L^2(\mathbb{G})$ given by 
$$\hat{\pi}'(x)=\hat{J}_{\mathbb{G}}\hat{\pi}(\hat{J}_{\mathbb{B}}x\hat{J}_{\mathbb{B}})\hat{J}_{\mathbb{G}}.$$
the space $\mathcal{I}$ is endowed with 
\begin{itemize}
    \item its natural $\sL(\GG)$ left action by composition,
    \item its natural $\sL(\BB)$ right action by composition,
    \item an $\sL(\BB)$-inner product given by $\prodscal{u}{v}_{\sL(\BB)}=u^*v,$ for all $u,v\in \mathcal{I}$.
    \item a *-morphism $\beta_{\mathcal{I}} : L^\infty(\GG)'\to \sL(\mathcal{I})$ given by $\beta_{\mathcal{I}}(m'(f))u=m(R(f))u$, for all $f\in \sA(\GG)$ and $u\in \mathcal{I}$.
\end{itemize}
With this structure the space $\mathcal{I}$ is a bicovariant $W^*$-bimodule (see \cite[Section 3.2]{Vaes}). 

Let $K$ be a $B$-Hilbert module endowed with a left $\sL(\BB)$-action. One can consider the space $\mathcal{I}\otimes_{\sL(\BB)}K$, which is a $B$-Hilbert module when it is endowed with the \textit{interior inner product} (\cite[Section 12.3]{Vaes}) as follows. Let $u,v\in \mathcal{I}$ and $x,y\in K$. The interior tensor product is given by 
$$\prodscal{u\otimes x}{v\otimes y}_B=\prodscal{x}{\prodscal{u}{v}_{\sL(\BB)}\cdot y}_B. $$

Now, following Vaes' induction procedure, we set $K=L^2(\GG)\hat\otimes\sV$, as in Remark \ref{rmk:action_on_the_false}. Vaes builds a bicovariant $B$-correspondence 
$$\overset{L^\infty(\GG)'}{\hphantom{\boxed{I}}_{\sL(\GG)}\boxed{\mathcal{I}\otimes_{\mathcal{L}(\mathbb{B})}(L^2(\mathbb{G})\hat\otimes \sV)}\vphantom{ }_{\sL(\GG)}}.$$
 \begin{rmk}\label{rmk:explicit}
 On this balanced tensor product, the left action of $\sL(\GG)$ is done via the left action of $\sL(\GG)$ on $\mathcal{I}$. The right action  of $\sL(\GG)$ via its right action on $(L^2(\mathbb{G})\hat\otimes \sV)$, as specified in Remark \ref{rmk:action_on_the_false}. Finally the morphism $\beta : L^\infty(\GG)'\to \sL(\mathcal{I}\otimes_{\mathcal{L}(\mathbb{B})}(L^2(\mathbb{G})\hat\otimes \sV))$ is given by $\beta=(\beta_{\mathcal{I}}\hat\otimes \beta_{L^2(\mathbb{G})\hat\otimes \sV})\circ\Delta$. Specifically, let $f\in\sD(\GG), \xi\in\sA(\GG)$ and $v\in \sV$. We have 
\begin{itemize}
\item $\lambda(g)\cdot(\iota(f)\otimes \Lambda(\xi)\otimes v)=\iota(g*f)\otimes \Lambda(\xi)\otimes v,$
\item $(\iota(f)\otimes \Lambda(\xi)\otimes v)\cdot \lambda'(g) =\iota(f)\otimes \Lambda(\xi*g)\otimes v,$
\item $\beta(m'(g))(\iota(f)\otimes \Lambda(\xi)\otimes v)=\iota(R(g_{(2)})f)\otimes \Lambda(R(g_{(1)})\xi)\otimes v.$
\end{itemize}
 \end{rmk}
Then, using Proposition \ref{prop:vaes} we have the existence of a corepresentation  of $C_0^r(\GG)$ on a $B$-Hilbert module $\mathrm{Ind}~\sV$ such that there is an isomorphism of $B$-correspondences
$$ \overset{L^\infty(\GG)'}{_{\sL(\GG)}\boxed{\mathcal{I}\hat\otimes_{\mathcal{L}(\mathbb{B})}(L^2(\mathbb{G})\otimes \sV)}\vphantom{ }_{\sL(\GG)}}\cong \overset{L^\infty(\GG)'}{_{\sL(\GG)}\boxed{L^2(\GG)\hat\otimes\mathrm{Ind}~\sV}\vphantom{ }_{\sL(\GG)}}.$$ 
The aim of this section is therefore to establish an equivalence of corepresentations
$$\mathrm{Ind}~\sV \cong \mathcal{E}(\mathbb{G})\hat\otimes_{C^*(\BB)}\sV.$$
%\begin{rmk}We will clarify the morphisms of this structure in Remark \ref{rmk:explicit}. Before this, we need to establish the link between $\mathcal{I}$ and our induction module, which is done in Proposition \ref{map}.\end{rmk}

According to Proposition \ref{prop:regtocor}, there exists a structure of bicovariant $B$-correspondence
$\overset{L^\infty(\GG)'}{_{\sL(\GG)}\boxed{L^2(\mathbb{G})\otimes \mathcal{E}(\mathbb{G})\otimes_{\sD(\BB)}\sV}\vphantom{ }_{\sL(\GG)}}$. Specifically, let $g\in\sD(\GG)$, $\xi,f\in \sA(\GG)$ and $v\in \sV$. We have %Here we have $\alpha(g)(f\otimes v)=g*f\otimes v$, for $g\in\sD(\GG)$, $f\in \sE(\GG)$ and $v\in\sV$.
\begin{itemize}
\item $\lambda(g)\cdot(\Lambda(\xi)\otimes f\otimes v)=(\Lambda\otimes\id)(\hat{\Delta}(g)*(\xi\otimes f))\otimes v,$ \\
where $*$ refers to the product of $\sD(\GG)\hat\otimes \sD(\GG),$
\item $(\Lambda(\xi)\otimes f\otimes v)\cdot \lambda'(g) =\Lambda(\xi*g)\otimes f \otimes v,$
\item $\beta(m'(g))(\Lambda(\xi)\otimes f\otimes v)=\Lambda(R(g)\xi)\otimes f\otimes v.$
\end{itemize}
\begin{prop}\label{equiv} We have an equivalence of bicovariant $B$-correspondences
$$\overset{L^\infty(\GG)'}{_{\sL(\GG)}\boxed{\mathcal{I}\hat\otimes_{\mathcal{L}(\mathbb{B})}(L^2(\mathbb{G})\hat\otimes \sV)}\vphantom{ }_{\sL(\GG)}}\cong  \overset{L^\infty(\GG)'}{_{\sL(\GG)}\boxed{L^2(\mathbb{G})\hat\otimes\mathcal{E}(\mathbb{G})\hat\otimes_{C^*(\BB)} \sV}\vphantom{ }_{\sL(\GG)}}. $$
\end{prop}
\justify To prove this we need several results.
\begin{lm}\label{lm:multiplier}
Let $h\in \mathcal{D}(\mathbb{B})$. We have that $\hat\pi(h)\delta_{\mathbb{G}}^{\frac{1}{2}}$ is a well defined element of $M(\mathcal{D}(\mathbb{G}))$ and we have $\hat\pi(h)\delta_{\mathbb{G}}^{\frac{1}{2}}=\hat\pi(h\pi(\delta_{\mathbb{G}}^{\frac{1}{2}})).$
\end{lm}
\begin{proof}
First, observe that, since $\delta_{\mathbb{G}}^{\frac{1}{2}}$ is group-like,  $f\mapsto f\delta_{\mathbb{G}}^{\frac{1}{2}}$ is a bijective homomorphism of the algebra $\mathcal{D}(\mathbb{G})$. As a consequence this map extends to a map $M(\mathcal{D}(\mathbb{G}))\rightarrow M(\mathcal{D}(\mathbb{G}))$ defined for $m\in M(\mathcal{D}(\mathbb{G}))$ and $f\in \mathcal{D}(\mathbb{G})$ by
$$(m\delta_{\mathbb{G}}^{\frac{1}{2}})*f=(m*(f\delta_{\mathbb{G}}^{-\frac{1}{2}}))\delta_{\mathbb{G}}^{\frac{1}{2}}.$$
Then, setting $m=\hat\pi(h)$, we get
\begin{align*}
    (\hat\pi(h)\delta_{\mathbb{G}}^{\frac{1}{2}})*f&=(\hat\pi(h)*(f\delta_{\mathbb{G}}^{-\frac{1}{2}}))\delta_{\mathbb{G}}^{\frac{1}{2}}\\
&=\phi_\mathbb{B}(\pi(S^{-1}(f_{(2)}\delta_{\mathbb{G}}^{-\frac{1}{2}}))h)f_{(1)}\delta_{\mathbb{G}}^{-\frac{1}{2}}\delta_{\mathbb{G}}^{\frac{1}{2}}\\
&=\phi_\mathbb{B}(\pi(S^{-1}(f_{(2)}))h\sigma_\BB(\pi(\delta_{\mathbb{G}}^{\frac{1}{2}})))f_{(1)}\\
&=\hat\pi(h\pi(\delta_{\mathbb{G}}^{\frac{1}{2}}))*f
\end{align*}
where the last equality follows from the hypothesis we made in Remark \ref{rmk:hyp} which gives  $\sigma_{\mathbb{B}}(\pi(\delta_{\mathbb{G}}^{\frac{1}{2}}))=\pi(\delta_{\mathbb{G}}^{\frac{1}{2}})$.
\end{proof}

\begin{lm}
Let $h\in \mathcal{D}(\mathbb{B})$ and $\xi\in \mathcal{D}(\mathbb{G})$. We have that $\hat{\pi}'(\lambda'(h))\Lambda(\xi)=\Lambda(\xi)\cdot h.$
\end{lm}
\begin{proof}
Let $h\in \mathcal{D}(\mathbb{B})$. We recall that we have the polar decomposition of the operator $\hat{T}_\BB : \Lambda(f)\mapsto \Lambda(f^*)$ as $\hat{T}_\BB=\hat{J}_\BB\hat{\nabla}_{\mathbb{B}}^{\frac{1}{2}}=\hat{\nabla}_{\mathbb{B}}^{-\frac{1}{2}}\hat{J}_\BB$, so
$$\hat{J}_{\mathbb{B}}\lambda'(h)\hat{J}_{\mathbb{B}}=\hat{\nabla}_{\mathbb{B}}^{\frac{1}{2}}\lambda'(h^*)\hat{\nabla}_{\mathbb{B}}^{-\frac{1}{2}}.$$
 Recall also that $\hat{\nabla}_{\mathbb{B}}\Lambda(\eta)=\Lambda(S^2(\eta)\delta_{\mathbb{B}}^{-1})$. We thus have
 %Using the strong commutation of the operator extensions of $S^2$ and $\delta_\GG$ from Section \ref{sec:S2}, we obtain
$$\hat{J}_{\mathbb{B}}\lambda'(h)\hat{J}_{\mathbb{B}}=\lambda(|S|(h)^*\delta_{\mathbb{B}}^{\frac{1}{2}}), $$
where $|S|$ is the automorphism introduced in Proposition \ref{prop:sbar}.
%Note that one can express explicitly $|S|(f)=\delta^{1/4}_{\mathbb{G}}(\delta_{\hat{\mathbb{G}}}^{-1/4}*f*\delta^{1/4}_{\hat{\mathbb{G}}})\delta_{\mathbb{G}}^{-1/4}$
Of course the same result stay true if we replace $\BB$ by $\GG$. We use in the next calculation that the  automorphisms $|S|$ are intertwined by $\hat\pi$ and that $|S|(\delta_{\mathbb{B}})=\delta_{\mathbb{B}}.$ We also have
\begin{align*}
(h\delta_{\mathbb{B}}^{\frac{1}{2}})^*&=\overline{S(h\delta_{\mathbb{B}}^{\frac{1}{2}})}\delta_{\mathbb{B}}\\
&=h^*\delta_{\mathbb{B}}^{-\frac{1}{2}}.
\end{align*}
One can now calculate

\begin{align*}
    \hat{\pi}'(\lambda'(h))&=\hat{J}_{\mathbb{G}}\hat{\pi}(\lambda(|S|(h)^*\delta_{\mathbb{B}}^{-\frac{1}{2}}))\hat{J}_{\mathbb{G}}\\
    &=\hat{J}_{\mathbb{G}}\hat{\pi}(\lambda(|S|(h\delta_{\mathbb{B}}^{\frac{1}{2}})^*))\hat{J}_{\mathbb{G}}\\
    &=\hat{J}_{\mathbb{G}}\lambda(|S|(\hat{\pi}(h\delta_{\mathbb{B}}^{\frac{1}{2}}))^*)\hat{J}_{\mathbb{G}}\\
    &=\hat{J}_{\mathbb{G}}\lambda(|S|(\hat{\pi}(h\delta_{\mathbb{B}}^{\frac{1}{2}}))^*\delta_{\mathbb{G}}^{\frac{1}{2}}\delta_{\mathbb{G}}^{-\frac{1}{2}})\hat{J}_{\mathbb{G}}\\
    &=\hat{J}_{\mathbb{G}}\lambda(|S|(\hat{\pi}(h\delta_{\mathbb{B}}^{\frac{1}{2}})\delta_{\mathbb{G}}^{-\frac{1}{2}})^*\delta_{\mathbb{G}}^{-\frac{1}{2}})\hat{J}_{\mathbb{G}}\\
    &\overset{(*)}{=}\hat{J}_{\mathbb{G}}\lambda(|S|(\hat{\pi}(h\gamma)^*\delta_{\mathbb{G}}^{-\frac{1}{2}})\hat{J}_{\mathbb{G}}\\
    &=\hat{J}_{\mathbb{G}}\hat{J}_{\mathbb{G}}\lambda'(\hat{\pi}(h\gamma))\hat{J}_{\mathbb{G}}\hat{J}_{\mathbb{G}}\\
    &=\lambda'(\hat{\pi}(h\gamma^{-1})).
\end{align*}
Where for $(*)$ we used Lemma \ref{lm:multiplier}. The result follows from the definition of the right action.
\end{proof}

\begin{lm}\label{injection}
The map $\rho_{\bullet}$ from Proposition \ref{map} defines an injection with dense image $\mathcal{E}(\mathbb{G})\rightarrow  \mathcal{I}$ (with respect to the weak topology of $B(L^2(\mathbb{B}),L^2(\mathbb{G})$). Its image is denoted $\mathcal{I}_0$.
\end{lm}

\begin{proof}
Let $f\in \mathcal{D}(\mathbb{G})$, $h\in \mathcal{D}(\mathbb{B})$ and $\eta\in \mathcal{D}(\mathbb{B})$. We have
\begin{align*}
    \rho_f(\lambda'(h)\Lambda(\eta))&=\Lambda(f\cdot(\eta*h))\\
    &=\Lambda((f\cdot\eta)\cdot h)\\
    &=\hat{\pi}'(\lambda'(h))\rho_f(\Lambda(\eta)).
\end{align*}
Thus the operator $\rho_f$ belongs to $\mathcal{I}$. It remains to show that the image of $\mathcal{E}(\mathbb{G})$ in $\mathcal{I}$ is dense. Let $\xi \in L^2(\mathbb{G})$ and $\eta \in L^2(\mathbb{B})$. Suppose we have
$$\prodscal{\Lambda(\xi)}{\rho_f(\Lambda(\eta))}=0$$
for all $f\in \mathcal{E}(\mathbb{G})$.
Let then $u\in \mathcal{I}$, we want to obtain that $\prodscal{\xi}{u(\Lambda(\eta))}=0$. Let $\varepsilon> 0$, there exist 
\begin{enumerate}
    \item $a\in \mathcal{A}(\mathbb{B})$ s.t. $\|\Lambda(\eta-a)\|_{L^2(\mathbb{B})}\leq \varepsilon$ (density of $\mathcal{A}(\mathbb{B})$),
    \item $b\in \mathcal{A}(\mathbb{B})$ s.t. $\|\Lambda(b*a-a)\|_{L^2(\mathbb{B})}\leq \varepsilon$ (essentialness),
    \item $c\in \mathcal{A}(\mathbb{G})$ s.t. $\|u(\Lambda(b))-\Lambda(c))\|_{L^2(\mathbb{G})}\leq \varepsilon$ (density of $\mathcal{A}(\mathbb{G})$).
\end{enumerate}
Now, there exist $k_1, k_2, k_3 > 0$ (depending only on the norms of $\Lambda(\xi)$, $\Lambda(\eta)$ and $u$ such that 
\begin{enumerate}
    \item $\mid\prodscal{\Lambda(\xi)}{u(\Lambda(\eta))}-\prodscal{\Lambda(\xi)}{u(\Lambda(a))}\mid\leq  k_1\varepsilon,$
    \item $\mid\prodscal{\Lambda(\xi)}{u(\Lambda(a))}-\prodscal{\Lambda(\xi)}{u(\Lambda(b*a))}\mid\leq k_2\varepsilon$,~  and we note that $u(\Lambda(b*a))=\lambda'(a)u(\Lambda(b))$;
    \item $\mid\prodscal{\Lambda(\xi)}{u(\Lambda(b))\cdot \lambda(a)}-\prodscal{\Lambda(\xi)}{\Lambda(c\cdot \eta)}\mid\leq k_3\varepsilon$,
\end{enumerate}
Finally, since $\prodscal{\xi}{c\cdot \eta}=0$ we have 
$$\mid \prodscal{\xi}{u(\eta)}\mid~\leq (k_1+k_2+k_3)\varepsilon,$$
So $\prodscal{\xi}{u(\eta)}=0$ and we are done.
\end{proof}

\begin{lm}\label{regular}
Let $\sV$ be a representation of $C^*_u(\mathbb{G})$ on any Hilbert module. One can endow $\sV$ with its von Neumann bornology  and consider the bornological space
$$\sV^{\infty}=\mathcal{D}(\mathbb{G})\hat\otimes_{\mathcal{D}(\mathbb{G})}\sV, $$
equipped with the left convolution action of $\mathcal{D}(\mathbb{G})$ is a bornological $\mathcal{D}(\mathbb{G})$-module  and defines a dense subspace of $\sV$.
\end{lm}
\begin{proof}
First, from the associativity of the bornological tensor product we have that
$$\mathcal{D}(\mathbb{G})\hat\otimes_{\mathcal{D}(\mathbb{G})}(\mathcal{D}(\mathbb{G})\hat\otimes_{\mathcal{D}(\mathbb{G})}\sV)=\mathcal{D}(\mathbb{G})\hat\otimes_{\mathcal{D}(\mathbb{G})}\sV,$$ and thus $\sV^{\infty}$ is a bornological $\mathcal{D}(\mathbb{G})$-module. 

Now consider the linear map $\mathcal{D}(\mathbb{G})\otimes_{\mathcal{D}(\mathbb{G})}\sV\to C^*_u(\mathbb{G})\otimes\sV$ defined by $f\otimes v\mapsto \lambda(f)\otimes v$. This map is bounded since bounded subspaces of $\sV$ are precisely bounded subspaces of $\sV$ with respect to its Hilbert topology. Furthermore this map leads to an injective map $\mathcal{D}(\mathbb{G})\hat\otimes_{\mathcal{D}(\mathbb{G})}\sV\to C^*_u(\mathbb{G})\hat\otimes_{C^*_u(\mathbb{G})}\sV\cong \sV$ which has dense range.

%It thus leads to a bounded injection $\sV^\infty\to \sV$. Too check that this injection has dense image, observe that from the essentialness of $\sV$ one can consider the map 

%The inclusion is given by the map determined by $f\otimes v\in \mathcal{D}(\mathbb{G})\otimes_{\mathcal{D}(\mathbb{G})}\sV \mapsto f\cdot v$. The associativity of the convolution ensures that it is an intertwiner and the density comes from the fact that $\mathcal{D}(\mathbb{G})$ is dense in $C^*_u(\mathbb{G})$ and that $\sV$ is a non-degenerate $C^*_u(\mathbb{G})$-module and thus essential.\\
\end{proof}
\begin{rmk}
Using the duality between modules and comodules at the bornological level, we obtain that $\sV^{\infty}$ is also a comodule.
\end{rmk}
\begin{lm}
Let $x,y\in\sA(\GG)$. We have 
$$xy=\phi_\GG(xy^{(1)})y^{(2)},$$
where the notation $y^{(1)}, y^{(2)}$ refers to the legs of the coproduct $\hat{\Delta}(y)$.
\end{lm}
\begin{proof}
By duality, it is enough to show that $f*g=\epsilon(f*g_{(1)})g_{(2)}$ for all $f,g\in\sA(\GG)$. We have 
\begin{align*}
    \epsilon(f*g_{(1)})g_{(2)}&=\phi_\GG(S^{-1}(g_{(1)})f)g_{(2)}\\
    &=f*g.
\end{align*}
\end{proof}
\begin{lm}\label{lm:technique_for_action}Let $\sV^\infty$ be a $\mathcal{D}(\mathbb{B})$-module and let
$h\in  \mathcal{D}(\mathbb{B})$, $\eta \in  \mathcal{A}(\mathbb{G})$ and $w \in \sV^{\infty}$. We have
$$\lambda(h)\cdot(\Lambda(\eta)\otimes w)=\Lambda(\eta_{(2)})\btimes \pi(S^{-1}(\eta_{(1)}))h\cdot w,$$
where ``$~\cdot~$'' on the left hand side stands for the diagonal action of $\mathcal{D}(\mathbb{B})$ on $\sV^{\infty}$.
\end{lm}
\begin{proof}We use $\alpha$ to denote the action of $\mathcal{D}(\mathbb{B})$ on $\sV^{\infty}$. We have
\begin{align*}
\lambda(h)\cdot(\Lambda(\eta)\otimes w)&=\lambda(h)\cdot(\Lambda(\eta)\otimes w)\\ &=(\lambda\circ\hat{\pi}\btimes \alpha)(\hat{\Delta}(h))(\Lambda(\eta)\otimes w)\\
&=\lambda(\hat{\pi}(h^{(1)}))\Lambda(\eta)\hat\otimes h^{(2)}\cdot w\\
&=\phi_\BB(\pi(S^{-1}(\eta_{(1)}))h^{(1)})\Lambda(\eta_{(2)})\btimes  h^{(2)}\cdot w\\
&=\Lambda(\eta_{(2)})\btimes \pi(S^{-1}(\eta_{(1)}))h\cdot w.
\end{align*}

\end{proof}

\begin{proof}[Proof of Proposition \ref{equiv}]
Lemma \ref{injection} allows us to consider the linear map
\begin{align*}
\Psi : \mathcal{A}(\mathbb{G})\otimes \mathcal{D}(\mathbb{G})\otimes \sV^\infty\rightarrow &\  \mathcal{I}\otimes(L^2(\mathbb{G})\otimes \sV) \\
\xi\otimes f\otimes v \mapsto& (\iota\hat\otimes\Lambda)(\Delta(\xi)(f\otimes 1)) \otimes v
\end{align*}
where $\iota$ stands for the injection $\sD(\GG)\to\mathcal{I}$ from Proposition \ref{map}. 

%We claim that the map $\Psi$ defined above is compatible with the inner products of $\mathcal{A}(\mathbb{G})\otimes \mathcal{E}(\mathbb{G})\otimes_{\sD(\BB)}\otimes\sV\infty$ and $\mathcal{I}\otimes_{\mathcal{L}(\mathbb{B})}(L^2(\mathbb{G})\otimes \sV)$. Thus it descends to a unitary map between balanced tensor product $\tilde{\Psi} : \mathcal{A}(\mathbb{G})\otimes \mathcal{D}(\mathbb{G})\otimes_{\sD(\BB)} \sV^\infty\rightarrow &\  \mathcal{I}\otimes_{\mathcal{L}(\mathbb{B})}(L^2(\mathbb{G})\otimes \sV)$.
Let $\xi, \eta\in\sA(\GG)$, $f,g\in\sD(\GG)$ and $v,w\in \sV^\infty$. We consider elements $[\iota(\xi_{(1)}f)\otimes\Lambda(\xi_{(2)}) \otimes v]$ and $[\iota(\eta_{(1)}g)\otimes\Lambda(\eta_{(2)}) \otimes w]$ of $\mathcal{I}\otimes_{\mathcal{L}(\mathbb{B})}(L^2(\mathbb{G})\otimes \sV)$ (where we us the notation $[ ~\cdot~]$ to refer to a class of elements in the balanced tensor product). Using the Lemma \ref{lm:technique_for_action} in the equality $(*)$ below, and the definition of the interior inner product, we obtain
\begin{align*}
    &\prodscal{[\iota(\xi_{(1)}f)\otimes\Lambda(\xi_{(2)}) \otimes v]}{[\iota(\eta_{(1)}g)\otimes\Lambda(\eta_{(2)}) \otimes w]}_{\mathcal{I}\otimes_{\mathcal{L}(\mathbb{B})}(L^2(\mathbb{G})\otimes \sV)}\\
    &= \prodscal{\Lambda(\xi_{(2)}) \otimes v}{\lambda(\prodscal{ \xi_{(1)}f}{\eta_{(1)}g }_{\sD(\BB)})\cdot(\Lambda(\eta_{(2)}) \otimes w)}_{L^2(\mathbb{G})\otimes \sV}\\
    &=\prodscal{\Lambda(\xi_{(2)}) \otimes v}{\phi_{\mathbb{G}}(\overline{\xi_{(1)}f}\eta_{(1)}g_{(1)})\lambda(\pi(\eta_{(2)}g_{(2)}\gamma)) \cdot(\Lambda(\eta_{(3)}) \otimes w)}_{L^2(\mathbb{G})\otimes \sV}\\
    &\stackrel{(*)}{=}\prodscal{\Lambda(\xi_{(2)}) \otimes v}{\phi_{\mathbb{G}}(\overline{\xi_{(1)}f}\eta_{(1)}g_{(1)})\Lambda(\eta_{(4)})\otimes  (\pi(S^{-1}(\eta_{(3)}))\pi(\eta_{(2)}g_{(2)}\gamma))\cdot w}_{L^2(\mathbb{G})\otimes \sV}\\
    &=\prodscal{\Lambda(\xi_{(2)}) \otimes v}{\phi_{\mathbb{G}}(\bar{f}~\overline{\xi_{(1)}}\eta_{(1)}g_{(1)})\Lambda(\eta_{(2)}) \otimes (\pi (g_{(2)}\gamma)\cdot w)}_{L^2(\mathbb{G})\otimes \sV}\\
    &=\phi_{\mathbb{G}}(\overline{\xi_{(2)}}\eta_{(2)})\prodscal{v}{\phi_{\mathbb{G}}(\bar{f}~\overline{\xi_{(1)}}\eta_{(1)}g_{(1)})  (\pi (g_{(2)}\gamma)\cdot w)}_{\sV}\\
    &=\prodscal{\Lambda(\xi) \otimes v}{\Lambda(\eta) \otimes \prodscal{f}{g}_{\sD(\BB)}\cdot w}_{L^2(\mathbb{G})\otimes \sV}\\
    &=\prodscal{\Lambda(\xi)\otimes [f\otimes v]}{\Lambda(\eta)\otimes [g\otimes w]}_{L^2(\mathbb{G})\otimes \mathcal{E}(\mathbb{G})\otimes_{\sD(\BB)} \sV}.
\end{align*}
In particular, this shows that elements in the kernel of the quotient  $\mathcal{A}(\mathbb{G})\otimes \mathcal{D}(\mathbb{G})\otimes \sV^\infty\rightarrow \mathcal{A}(\mathbb{G})\otimes \mathcal{D}(\mathbb{G})\otimes_{\sD(\BB)} \sV$ have null image in $\mathcal{I}\otimes_{\mathcal{L}(\mathbb{B})}(L^2(\mathbb{G})\otimes \sV)$ so the map $\Psi$ descends to a unitary map to the balanced tensor product.

Now we can consider the associated map 
$$\tilde\Psi : L^2(\mathbb{G})\otimes \mathcal{E}(\mathbb{G})\otimes_{C^*(\BB)} \sV\rightarrow \  \mathcal{I}\otimes_{\mathcal{L}(\mathbb{B})}(L^2(\mathbb{G})\otimes \sV).$$ 
Let us check that $\Psi$ intertwines the bicovariant $B$-correspondence structure. We start with the left action of $\sL(\GG)$. Let $g\in \mathcal{D}(\mathbb{G})$ and $\xi\otimes f\otimes v\in\mathcal{A}(\mathbb{G})\otimes \mathcal{D}(\mathbb{G})\otimes\sV^\infty$. We have
\begin{align*}
\lambda(g)\cdot (\Lambda(\xi)\otimes [f\otimes v])&=
    ((\lambda\otimes\id)(\hat{\Delta}(g)\otimes\id))(\Lambda(\xi)\otimes [f\otimes v]) \\
    &=\phi_{\mathbb{G}}(S^{-1}(\xi_{(1)}f_{(1)})g)\Lambda(\xi_{(2)})\otimes [f_{(2)}\otimes v],
\end{align*}
and
\begin{align*}
    \phi_{\mathbb{G}}(S^{-1}(\xi_{(1)}f_{(1)})g)\tilde\Psi(\Lambda(\xi_{(2)})\otimes [f_{(2)}\otimes v])&=\phi_{\mathbb{G}}(S^{-1}(\xi_{(1)}f_{(1)})g)[\iota(\xi_{(2)}f_{(2)})\otimes\Lambda(\xi_{(3)})\otimes v]\\
    &=[\iota(g*\xi_{(1)}f)\otimes \Lambda(\xi_{(2)})\otimes v]\\
    &=\lambda(g)\cdot [\iota(\xi_{(1)}f)\otimes \Lambda(\xi_{(2)})\otimes v]
\end{align*}
For the right action of $\sL(\GG)'$ consider again $g\in \mathcal{D}(\mathbb{G})$ and $\xi\otimes f\otimes v\in\mathcal{A}(\mathbb{G})\otimes \mathcal{D}(\mathbb{G})\otimes\sV^\infty$. We have
\begin{align*}
\tilde\Psi(\Lambda(\xi*g)\otimes [f\otimes v])\cdot\lambda(g)&=[\iota(\xi_{(1)}f)\otimes\Lambda(\xi_{(2)}*g)\otimes v]\\
&=\Psi(\Lambda(\xi)\otimes [f\otimes v])\cdot \lambda'(g),
\end{align*}
where we use Lemma \ref{lem:convolution_coproduct_compatibility}. Let now $g\in L^\infty(\GG)'$. We have
\begin{align*}
    \tilde\Psi(\beta(g)(\Lambda(\xi)\otimes [f\otimes v]))&=\tilde\Psi(\Lambda(R(g)\xi)\otimes [f\otimes v])\\
    &=\iota(R(g_{(2)})\xi_{(1)}f)\otimes\Lambda(R(g_{(1)})\xi_{(2)})\\
    &=\beta(g)\tilde\Psi(\Lambda(\xi)\otimes [f\otimes v]),
\end{align*}
where we use that $R$ is an anti coalgebra homomorphism.

We finish with the surjectivity of $\tilde\Psi$. Let $f\in\sD(\GG), g\in\sA(\GG)$ and $v\in\sV^\infty$ and consider the element 
$\iota(f)\otimes[\Lambda(g)\otimes v]$ of $\mathcal{I}\otimes_{\mathcal{L}(\mathbb{B})}(L^2(\mathbb{G})\otimes \sV)$. We observe that the element $[\Lambda(g_{(2)})\otimes S^{-1}(g_{(1)})f\otimes v]$ of $L^2(\mathbb{G})\otimes \mathcal{E}(\mathbb{G})\otimes_{C^*(\BB)}\sV$ is an antecedent of  $\iota(f)\otimes[\Lambda(g)\otimes v]$ for $\tilde\Psi$. We conclude with a density argument.
\end{proof}
\begin{thm}
The representations $ {\rm Ind}  ~\sV$ and $\mathcal{E}(\mathbb{G})\otimes_{C^*(\BB)}\sV$ are equivalent.
\end{thm}
\begin{proof}
	From Proposition \ref{equiv} and Proposition \ref{prop:vaes} we obtain an isomorphism of bicovariant correspondence
	$$L^2(\mathbb{G})\otimes (\mathcal{E}(\mathbb{G})\otimes_{C^*(\BB)} \sV)\cong L^2(\mathbb{G})\otimes {\rm Ind}  ~\sV.$$
	The result follows.
\end{proof}
\section{Parabolic Induction}

In this section we give an explicit Rieffel induction module associated to the functor of parabolic induction for complex semi-simple quantum groups. In particular we show that our induction functor coincides with the classical definition in the case of parabolic induction \cite{Arano,yuncken-voigt}.  Finally, we give a geometric presentation of this parabolic induction module, in a similar way to what Clare did in \cite{Clare} for classical semi-simple Lie groups.
\subsection{Preliminaries}
We follow the notations and conventions of \cite{yuncken-voigt}.
Let $\mathfrak{g}$ be a complex semisimple Lie algebra and let $G_q$ be the associated simply connected complex semisimple quantum group and $K_q$ its maximal compact quantum subgroup, with its multiplicative unitary $W\in M(\mathcal{A}(\hat{K}_q)\otimes\mathcal{A}(K_q))$. We write $\mathcal{U}_q(\mathfrak{g})$ for the associated quantized enveloping algebra and $\mathcal{U}_q^{\mathbb{R}}(\mathfrak{k})$ for the same algebra equipped with the involution $*$, seen as the complexification of the quantized enveloping algebra of the compact form $K_q$. We recall that the algebra of representative functions on the Drinfeld double $G_q=K_q\bowtie \hat{K}_q$ is defined by
$$\mathcal{A}(G_q)=\mathcal{A}(K_q)\otimes \mathcal{A}(\hat{K}_q),$$ with coproduct 
$$\Delta_{G_q}(a\otimes f)=W^{-1}_{32}(a_{(1)}\otimes f_{(1)}\otimes a_{(2)}\otimes f_{(2)})W_{32}.$$
This idea is originally due to Podles and Woronowicz \cite{Wo}.  For a complete  discussion see \cite[Definition 3.18]{yuncken-voigt}.

Let $(u_{ij}^{\sigma})\in \mathcal{A}(K_q)$ denote the matrix coefficient associated to a weight basis of an irreducible representation $\sigma$ of $K_q$ and let $(\omega_{ij}^{\sigma})\in \mathcal{A}(\hat{K_q})$ denote the elements of the dual basis. We have
$$W=\sum_{i,j,\sigma}u_{ij}^{\sigma}\otimes \omega_{ij}^{\sigma},~~~~W^{-1}=\sum_{i,j,\sigma}S(u_{ij}^{\sigma})\otimes \omega_{ij}^{\sigma}, $$
where the sums run over all equivalent classes of irreducible representations. In practice we only write $W=u_{ij}^{\sigma}\otimes \omega_{ij}^{\sigma}$.

The multiplier Hopf algebra $\sA(G_q)$ is equipped with a left Haar state $\phi_{K_q}\otimes \psi_{\hat{K}_q}$, where $\phi_{K_q}$ is the Haar state on $\mathcal{A}(K_q)$ and $\psi_{\hat{K}_q}$ the right Haar state on $\mathcal{A}(\hat{K}_q)$. Note that this is also a right Haar state, so that $\delta_{G_q}=1$.

Let $\textbf{P}$ be the weight lattice associated to $\mathfrak{g}$ and  $\mathcal{U}_q^{\mathbb{R}}(\mathfrak{t})=\text{span}\{K_{\lambda}, \lambda\in \textbf{P}\}$. For each $\mu\in \textbf{P}$ we define $e^{\mu} \in \mathcal{U}_q^{\mathbb{R}}(\mathfrak{t})'$ by 
$$e^{\mu}(K_{\lambda})=q^{(\lambda,\mu)}.$$
In this way we can identify the algebra of functions on the torus subgroup $T$ of $K_q$ as 
$$\mathcal{A}(T)=\text{span}\{e^{\mu}, \mu\in \textbf{P}\}\subset \mathcal{U}_q^{\mathbb{R}}(\mathfrak{t})',$$
where $\textbf{P}$ is the weight lattice.

\begin{rmk} The classical torus subgroup $T$ is naturally identified with $\mathrm{spec}(\mathcal{A}(T))$ and we note that for any $\lambda\in \mathfrak{t}^*$ we obtain a character of the $*$-algebra $\mathcal{A}(T)$ by 
$$(K_{\lambda},e^{\mu})=q^{i(\lambda,\mu)}. $$
This yields an identification $T\cong i(\mathfrak{t}^*/\frac{2\pi}{log(q)}Q^{\vee})$ where $Q^{\vee}=\mathrm{Hom}(\textbf{P},\mathbb{Z})$ is the coroot lattice, see \cite[Section 5.11]{yuncken-voigt}. We will not use this identification in what follows.
\end{rmk}

We define the restriction map $\pi : \mathcal{A}(K_q)\rightarrow \mathcal{A}(T)$
via 
$$\pi(a)=a|_{ \mathcal{U}_q^{\mathbb{R}}(\mathfrak{t})}.$$

The Borel subgroup $B_q=T\bowtie \hat{K_q}$ is defined via $\mathcal{A}(B_q)=\mathcal{A}(T)\otimes \mathcal{A}(\hat{K}_q)$ (see \cite[section 4.7]{yuncken-voigt}) with coproduct 
$$\Delta(a\otimes f)=\widetilde{W}^{-1}_{32}(a_{(1)}\otimes f_{(1)}\otimes a_{(2)}\otimes f_{(2)})\widetilde{W}_{32}, $$
twisted by the bicharacter $\widetilde{W}=(\pi\otimes \text{id})(W)$. It is a closed subgroup of $G_q$ with restriction map
$$\pi\otimes \text{id} : \mathcal{A}(G_q)\rightarrow \mathcal{A}(B_q).$$
We denote by $\phi_T$ the Haar functional on $\mathcal{A}(T)$. The functional $\phi_T\otimes\psi_{\hat{K}_q}$, which will be denoted $\phi_{B_q}$, is left invariant on $\mathcal{A}(B_q)$. We have
 $$\delta_{B_q}=\textbf{1}\otimes K_{-4\rho},$$
 for the associated modular element (see the proof of \cite[Propopistion 4.19]{yuncken-voigt}).
We thus obtain our conditional expectation
 $E : \mathcal{D}(G_q)\rightarrow \mathcal{D}(B_q)$, $E(a\otimes f)=\pi(a)\otimes fK_{-2\rho}$, for all $a\otimes f$ in $\mathcal{D}(K_q)\otimes \mathcal{D}(\hat{K}_q)$. 
 
 Let $(\mu,\lambda)\in \textbf{P}\times \mathfrak{t}^*$. We recall that the principal series representation associated to $(\mu,\lambda)$ is defined to be the space 
 $$\mathrm{Ind}_{B_q}^{G_q}\mathbb{C}_{\mu,\lambda}=\{\xi\in M(\mathcal{A}(G_q))~|~(\text{id}\otimes \pi_{B_q})\Delta_{G_q}(\xi)=\xi\otimes (e^{\mu}\otimes K_{2\rho+\lambda})\},$$
 with a coaction induced by the comultiplication of $G_q$.
 The notation $\Ind_{B_q}^{G_q}\CC_{\mu,\lambda}$ here is inspired by analogy with the classical induction procedure, see \cite[Section 6.4.2]{yuncken-voigt}.  Our goal here is to show that this coincides with the induction functor which we have introduced here.
 
 %We also note that the convolution algebra can be explicitly defined in  as $\mathcal{D}(B_q)=\mathcal{D}(T)\bowtie \mathcal{D}(\hat{K}_q)$ with twisted product given for $x\bowtie u$ and$~y\bowtie v\in \mathcal{D}(T)\bowtie \mathcal{D}(\hat{K}_q)$ by $$(x\bowtie u)(y\bowtie v)=x(y_{(1)},\pi(u_{(1))})y_{(2)}\bowtie u_{(1)}(y_{(3)},S^{-1}(\pi(u_{(3)})))v,$$where we consider here the coalgebra structures of $\mathcal{D}(T)$ and $\mathcal{D}(\hat{K}_q)$ and the natural pairing between the Hopf algebras $\mathcal{D}(T)$ and $\mathcal{D}(\hat{T})$. The coproduct on $\mathcal{D}(B_q)$ is simply the untwisted coproduct of $\mathcal{D}(T)\otimes \mathcal{D}(\hat{K}_q)$.
 
 %For the discrete quantum group $\mathcal{A}(\hat{K}_q)$ we denote by $\widehat{\textbf{1}_{K_q}}$ the element corresponding to the Haar state $\phi_{K_q} \in \mathcal{D}(K_q)$ (where here  $\mathcal{D}(K_q)$ is seen as a subspace of $\mathcal{A}(K_q)'$). Under the identification $\mathcal{D}(\hat{K}_q)\cong \mathcal{A}(\hat{K}_q)$ as linear spaces, $\widehat{\textbf{1}_{K_q}}$ corresponds to the identity element for the convolution in $\mathcal{D}(\hat{K}_q)$.
 Finally, we notice that $B_q$ is amenable and that the condition $\sigma_{B_q}(\pi(\delta_{G_q}))$ is trivially verified since $\delta_{G_q}=1$.

\subsection{{The quotient map}}

In the classical case, with $G=KAN$, principal series representations are induced from characters of the the Borel subgroup $B=MAN$. Explicitly, we choose first a character  
$\mu$ of $M$ and $\lambda$ of $A$ and then the identification $MA=B/N$ allows us to extend $\mu\otimes\lambda$ to a character of $B$. In this way we obtain the principal series representation 
$$\text{Ind}_B^G~\mu\otimes\lambda.$$
In the quantum case we do not have an analog for the subgroup $N$. But, as we now explain, we do have a ``quotient'' map
$$\widehat{K_q} \twoheadrightarrow A_q.$$

Let us make this explicit. There are two versions of the map  $\pi_T$. Firstly, with the canonical identification of $*$-algebras $\mathcal{A}(K_q)=\mathcal{D}(\widehat{K_q})$ and $\mathcal{A}(T)=\mathcal{D}(A_q)$, one can consider 
$$\pi_T~:~\mathcal{D}(\widehat{K_q})\rightarrow \mathcal{D}(A_q), $$
which is a $*$-morphism and comes with its dual morphism $\hat{\pi}_T : \sA(A_q)\to M(\sA(\widehat{K_q}))$. Secondly, using the identifications of vector spaces $\mathcal{A}(\widehat{K_q})\cong \mathcal{D}(\widehat{K_q})$ and $\mathcal{A}(A_q)\cong\mathcal{D}(A_q)$ the same map can be interpreted as a map
$$\pi_T : \mathcal{A}(\widehat{K_q})\rightarrow \mathcal{A}(A_q).$$
This is a conditional expectation is the sense of Proposition \ref{cond}, observing that $K_q$ and $T$ are unimodular. In particular $\pi(fK_{\lambda})=\pi(f)K_{\lambda}$ for all $f\in \mathcal{A}(\widehat{K_q})$, $\lambda\in \textbf{P}$. This is the map $\pi_T : \mathcal{A}(\widehat{K_q})\rightarrow \mathcal{A}(A_q)$ that we call the \textit{quotient map}.
This map has also the notable property 
$$\phi_{\widehat{K_q}}(f)=\phi_{A_q}(\pi_T(f)). $$
 Indeed we have for all $a \in \mathcal{A}(K_q)$
\begin{align*}
   \phi_{\widehat{K_q}}(\mathcal{F}_{K_q}(a))&=\epsilon_{K_q}(a)\\
   &=\epsilon_{T}(\pi_T(a))\\
   &=\phi_{A_q}(\pi_T(\mathcal{F}_{K_q}(a))).
\end{align*}
%The $*$-morphism $\pi_T : \mathcal{D}(\widehat{K_q})\rightarrow \mathcal{D}(A_q)$ makes $\mathcal{D}(A_q)$ into a left $\mathcal{D}(\widehat{K_q})$-module given by $f\cdot h= \pi_T(f)*h$, for $f\in \mathcal{D}(\widehat{K_q})$, $h\in \mathcal{D}(A_q)$. Let us consider the associated corepresentation of $\mathcal{A}(\widehat{K_q})$ on $\mathcal{A}(A_q)$ given by $$h\mapsto \hat{\pi}_T(h_{(1)})\otimes h_{(2)}. $$
\begin{rmk}
In the rest of this paper we extensively use Sweedler notation. Since one considers both $\mathcal{A}(K_q)$ and $\mathcal{A}(\widehat{K_q})$, this can be confusing. The convention is as follows. If we write $f\in \sA(H_q)$ or $f\in \sD(H_q)$ (where $H_q=G_q,~K_q,~\widehat{K_q},~T$ or $A_q$) then $f_{(1)}\otimes f_{(2)}$ always refers to the coproduct of $\sA(G_q)$.
\end{rmk}
\begin{lm}\label{lm:formula}
Let $f\in \mathcal{A}(\widehat{K_q})$. We have 
$$\pi_T(f_{(2)})\otimes f_{(1)}=\pi_T(f)_{(2)}\otimes\hat\pi_T(\pi_T(f)_{(1)}).$$
In particular this means that the map $\sA(A_q)\to M( \mathcal{A}(\widehat{K_q})\otimes\sA(A_q))$ given by $\pi_T(f)\mapsto f_{(1)}\otimes\pi_T(f_{(2)})$ is well defined.
\end{lm}
\begin{proof}

Let $f,g\in \mathcal{D}(\widehat{K_q})$. On the one hand we have
\begin{align*}
    \pi_T(g*f)&=\pi_T(f_{(2)})\phi_{\widehat{K_q}}(S^{-1}(f_{(1)})g)\\
    &=\pi_T(f_{(2)})(g,S^{-1}(f_{(1)})).
\end{align*}
And on the other hand
\begin{align*}
    \pi_T(g)*\pi_T(f)&=\pi_T(f)_{(2)}\phi_{A_q}(S^{-1}(\pi_T(f)_{(1)})\pi_T(g))\\
    &=\pi_T(f)_{(2)}(g,\hat\pi_T(S^{-1}(\pi_T(f)_{(1)})))\\
    &=\pi_T(f)_{(2)}(g,S^{-1}(\hat\pi_T(\pi_T(f)_{(1)})))
\end{align*}

One can thus identify the legs and we obtain 
$$\pi_T(f_{(2)})\otimes f_{(1)}=\pi_T(f)_{(2)}\otimes\hat\pi_T(\pi_T(f)_{(1)}).$$

%Through the *-morphism $\pi_T : \sD(\widehat{K_q})\to\sD(A_q)$, the space $\sD(A_q)$ can be seen as a $\sD(\widehat{K_q})$ module and $\pi_T : \sD(\widehat{K_q})\to\sD(A_q)$ is a $\sD(\widehat{K_q})$-module morphism. One can see that the dual $\sA(\widehat{K_q})$ comodule structure (as in Proposition ??), is given by $h\mapsto (\hat\pi_T\otimes\id)(\Delta_{A_q}(h))$, where $\Delta_{A_q}$ is the coproduct of $\sA(A_q)$. As a consequence it is more natural to consider $\sA(A_q)$ as a $\sA(\widehat{K_q})$-comodule through this coaction. We thus know that $\pi_T : \sD(\widehat{K_q})\to\sD(A_q)$ is a $\sA(\widehat{K_q})$-comodule morphism. This directly give the relation 
%$$(\id\otimes\pi_T)(\Delta_{\widehat{K_q}})(f)=(\hat{\pi}_T\otimes\id)(\Delta_{A_q}(\pi_T(f))),$$ for all $f\in \mathcal{A}(\widehat{K_q})$. 
\end{proof}
We denote by $\alpha_{A_q} : \sA(A_q)\to M(\mathcal{A}(\widehat{K_q}))\otimes \sA(A_q)$ the $\mathcal{A}(\widehat{K_q})$ coaction we obtain on $\sA(A_q)$. That is, for $h\in\mathcal{A}(A_q)$ we have 
$$\alpha_{A_q}(h)=\hat\pi_T(h_{(1)})\otimes h_{(2)},$$
and for $f\in \mathcal{A}(\widehat{K_q})$ one can also write
$$\alpha_{A_q}(\pi_T(f))=f_{(1)}\otimes \pi_T(f_{(2)}).$$

\subsection{{The parabolic induction module}}
The goal here is to build a Hilbert module which implements the parabolic induction functor. We define this module in this section as a balanced tensor product $\mathcal{E}(G_q)\otimes_{C^*(B_q)}C^*(L_q)$, where $\mathcal{E}(G_q)$ is the induction module built from the closed quantum subgroup $B_q$ and where we note $L_q=T\times A_q$.

\begin{lm}
The linear map $(\id\otimes\pi_T) : \sD(B_q)\to \sD(L_q)$
is a *-Hopf homomorphism.
\end{lm}
\begin{proof}
We first show that $(\id\otimes\hat\pi_T) : \sA(L_q)\to M(\sA(B_q))$ is a *-Hopf homomorphism, then we conclude with a duality argument. Before we start, we recall that $\hat\pi_T : \sA(A_q)\to M(\sA(\widehat{K_q}))$ is a Hopf *-morphism. We have that
\begin{align*}
&\Delta_{B_q}(a\otimes f)=a_{(1)}\otimes \omega_{ii}^{\sigma}f_{(1)}\omega_{rr}^{\nu}\otimes \pi_T(u_{ii}^{\sigma}S(u_{rr}^{\nu})) a_{(2)}\otimes f_{(2)},
\end{align*}
for all $a\otimes f\in \sA(B_q)$. Let $a\otimes h\in \sA(L_q)$. We have on the one hand
\begin{align*}
    ((\id\otimes\hat\pi_T)\otimes(\id\otimes\hat\pi_T))(\Delta_{L_q}(a\otimes h))&=a_{(1)}\otimes \hat\pi_T(h_{(1)})\otimes a_{(2)}\otimes \hat\pi_T(h_{(2)}).
\end{align*}
And on the other hand 
\begin{align*}
  \Delta_{B_q}(a\otimes \hat\pi_T(h))&=a_{(1)}\otimes \omega_{ii}^{\sigma}\hat\pi_T(h_{(1)})\omega_{rr}^{\nu}\otimes \pi_T(u_{ii}^{\sigma}S(u_{rr}^{\nu})) a_{(2)}\otimes \hat\pi_T(h_{(2)}),
\end{align*}
and since $\hat\pi_T$ maps $\sA(A_q)$ into the set of diagonal elements of $\mathcal{A}(\widehat{K_q})$, we obtain 
\begin{align*}
    \Delta_{B_q}(a\otimes \hat\pi_T(h))&=a_{(1)}\otimes \omega_{ii}^{\sigma}\hat\pi_T(h_{(1)})\otimes \pi_T(u_{ii}^{\sigma}S(u_{ii}^{\sigma})) a_{(2)}\otimes \hat\pi_T(h_{(2)})\\
    &=a_{(1)}\otimes \hat\pi_T(h_{(1)})\otimes a_{(2)}\otimes \hat\pi_T(h_{(2)}).
\end{align*}
Thus $(\id\otimes\hat\pi_T)$ is compatible with the coproducts. The *-algebra structure of $\sA(B_q)$ is not twisted so there is no difficulty to see that $(\id\otimes\hat\pi_T)$ is a *-algebra homomorphism. To conclude we just notice that since the pairing between $\sD(B_q)$ and $\sA(B_q)$ is defined leg by leg it is clear that the dual morphism of $\id\otimes\hat\pi_T$ is $\id\otimes\pi_T$.
\end{proof}

Let $(\mu,\lambda)\in  \textbf{P}\times \mathfrak{t}^*_q$.  One can build the one dimensional representation of $L_q$ on 
$\mathbb{C}_{\mu,\lambda}=\mathbb{C}_{\mu}\otimes\mathbb{C}_{\lambda}$ via
$$(\tau\otimes h)\cdot 1=\phi_T(e^{-\mu}\tau)\phi_{A_q}(K_{-\lambda} h),$$ 
for all $h\in \mathcal{D}(A_q),~ \tau\in \mathcal{D}(T).$
Since $\mathcal{D}(L_q)$ is essential, we have  $\mathcal{D}(L_q)\otimes_{\mathcal{D}(L_q)}\mathbb{C}_{\mu,\lambda}\cong \mathbb{C}_{\mu,\lambda}$. Furthermore since $\mathcal{D}(L_q)$ is a left $\mathcal{D}(B_q)$-module, one can consider the action of $\mathcal{D}(B_q)$ on $\mathcal{D}(L_q)\otimes_{\mathcal{D}(L_q)}\mathbb{C}_{\mu,\lambda}$, which happens to be exactly the character of $B_q$ associated to $(\mu,\lambda)$, according to the previous lemma. In particular this shows that such character can be factorized through the morphism $(\id\otimes\pi_T) : \mathcal{D}(B_q)\to \mathcal{D}(L_q)$.

%\begin{rmk}Any pure tensor in $\mathcal{D}(G_q)\otimes_{\sD(B_q)}\mathbb{C}_{\mu,\lambda}$ can be written in the form $(a\otimes \widehat{1_{K_q}})\otimes 1$, $a\in \sA(K_q)$.\end{rmk}
We now confirm that the classical definition of a parabolically induced representation agrees with the general induction method we developed.

\begin{lm}\label{lm:decross}
Let $a\otimes f \in \mathcal{A}(G_q)$. We have
%$$((\id\otimes\id)\otimes(\pi_T\otimes\pi_T))(\Delta_{G_q} $$
$$(a\otimes f)_{(1)}\otimes (\pi_T\otimes\pi_T)((a\otimes f)_{(2)})=a_{(1)}\otimes f_{(1)}\otimes (\pi_T\otimes\pi_T)(a_{(2)}\otimes f_{(2)}),$$
where $(a\otimes f)_{(1)}\otimes (a\otimes f)_{(2)}$ refers to the coproduct of $\mathcal{A}(G_q)$.
\end{lm}
\begin{proof}
Let $a\otimes f \in \mathcal{A}(G_q)$. We have
\begin{align*}
(a\otimes f)_{(1)}\otimes (\pi_T\otimes\pi_T)((a\otimes f)_{(2)})
&=
   a_{(1)}\otimes \omega_{ij}^{\sigma}f_{(1)}\omega_{rs}^{\nu}\otimes \pi_T(S(u_{ij}^{\sigma}) a_{(2)}u_{rs}^{\nu})\otimes \pi_T(f_{(2)})\\
   &=a_{(1)}\otimes \omega_{ii}^{\sigma}f_{(1)}\omega_{rr}^{\nu}\otimes \pi_T(u_{ii}^{\sigma}S(u_{rr}^{\nu})) \pi_T(a_{(2)})\otimes \pi_T(f_{(2)})\\
   &\overset{(*)}{=}a_{(1)}\otimes \omega_{ii}^{\sigma}\hat{\pi}_T(\pi_T(f_{(1)}))\omega_{rr}^{\nu}\otimes \pi_T(u_{ii}^{\sigma}S(u_{rr}^{\nu})) \pi_T(a_{(2)})\otimes \pi_T(f_{(2)})\\
   &=a_{(1)}\otimes f_{(1)}\otimes \pi_T(a_{(2)})\otimes \pi_T(f_{(2)}),
\end{align*}
where at equality $(*)$ we used Lemma \ref{lm:formula}.
\end{proof} 
We now consider the $\sD(B_q)$-inner product on $\sD(G_q)$, given by Proposition \ref{prodscal}. According to \cite[Lemma 4.17]{yuncken-voigt} we have $\delta_{B_q}=1\otimes K_{-4\rho}$.
\begin{lm}
Let $a\otimes f, b\otimes g \in \mathcal{D}(G_q)$. We have 
$$(\id\otimes\pi_T)(\prodscal{a\otimes f}{b\otimes g}_{\sD(B_q)})=\pi_T(a^**b)\otimes \pi_T(f^**g)K_{-2\rho}.$$
\end{lm}
\begin{proof}
$a\otimes f, b\otimes g \in \mathcal{D}(G_q)$. Using Remark \ref{rmk:prodscalexpress} we obtain
\begin{align*}
    (\id\otimes\pi_T)(\prodscal{a\otimes f}{b\otimes g}_{\sD(B_q)})&=(\id\otimes\pi_T)(\phi_{G_q}(\overline{(a\otimes f)}(b\otimes g)_{(1)})(\pi_T\otimes\id)((b\otimes g)_{(2)})(1\otimes K_{-2\rho}))\\
    &=\phi_{G_q}(\overline{(a\otimes f)}(b\otimes g)_{(1)})(\pi_T\otimes\pi_T)((b\otimes g)_{(2)})(1\otimes K_{-2\rho})\\
    &\overset{(*)}{=}\phi_{G_q}((\overline{a}\otimes \overline{f})(b_{(1)}\otimes g_{(1)}))\pi_T(b_{(2)})\otimes \pi_T(g_{(2)})(1\otimes K_{-2\rho})
    \\&=\pi_T(a^**b)\otimes \pi_T(f^**g)K_{-2\rho},
\end{align*}
where for the equality $(*)$ we use the previous lemma and that the involution on $\mathcal{A}(G_q)$ is leg-wise. For the last line we simply use that $\phi_{G_q}=\phi_{K_q}\otimes\psi_{\widehat{K_q}}$ and identify convolutions on each leg.
\end{proof}
%\begin{rmk}This lemma implies that the scalar product in $`\sD(G_q)\otimes_{B_q}\sD(L_q)$ is given by$$\prodscal{a\otimes f}{b\otimes g}_{\sD(L_q)}=\pi_T(a^**b)\otimes \pi_T(f^**g)K_{-2\rho},~\forall a\otimes f, b\otimes g \in \mathcal{D}(G_q)$$\end{rmk}
\begin{prop}
The unitary representations $\mathcal{D}(G_q)\otimes_{\sD(B_q)}\mathbb{C}_{\mu,\lambda}$ and $ {\rm Ind}_{B_q}^{G_q}\mathbb{C}_{\mu,\lambda}$ of $\sD(G_q)$ are isomorphic.
\end{prop}
\begin{proof}
We consider the map $ \Psi$ such that 
\begin{align*}
    \Psi~:~\mathcal{D}(G_q)&\longrightarrow \text{Ind}_{B_q}^{G_q}\mathbb{C}_{\mu,\lambda}\\
    (a\otimes f)&\longmapsto a*\hat{\pi}_T(e^{\mu})\otimes \phi_{\widehat{K}_q}(fK_{-\lambda-2\rho})K_{\lambda+2\rho}.
\end{align*}
We will show that this map is surjective, intertwines the $\sA(G_q)$ coactions and descends to the balanced tensor product $\mathcal{D}(G_q)\otimes_{\sD(B_q)}\mathbb{C}_{\mu,\lambda}$. Let $a\otimes f\in \mathcal{D}(G_q)$. We first show that $a*\hat{\pi}_T(e^{\mu})\otimes \phi_{\widehat{K}_q}(fK_{-\lambda+2\rho})K_{\lambda+2\rho}$ belongs to $\mathrm{Ind}_{B_q}^{G_q}\mathbb{C}_{\mu,\lambda}$. It is enough to show that $(\id\otimes \pi_T)(\Delta_{K_q}(a*\hat{\pi}_T(e^{\mu})))=(a*\hat{\pi}_T(e^{\mu}))\otimes e^{\mu}$. For this, since $e^\mu$ is group-like we have 
\begin{align*}
    (\id\otimes \pi_T)(\Delta_{K_q}(a*\hat{\pi}_T(e^{\mu})))&=\phi_T(e^{-\mu}\pi_T(a_{(3)}))a_{(1)}\otimes \pi_T(a_{(2)})\\
    &=\phi_T(e^{-\mu}\pi_T(a_{(2)}))a_{(1)}\otimes e^{\mu}.
\end{align*}

Next, let $a\otimes f, b\otimes g$ be in $\mathcal{D}(G_q)$ and consider the elements $[a\otimes f\otimes 1], [b\otimes g\otimes 1]$ of $\mathcal{D}(G_q)\otimes_{\sD(B_q)}\mathbb{C}_{\mu,\lambda}$. We have
\begin{align*}
   \prodscal{[(a\otimes f)\otimes 1]}{[(b\otimes g)\otimes 1]}&=\prodscal{a\otimes f}{b\otimes g}_{\sD(B_q)}\cdot 1\\
   &=(\id\otimes\pi_T)(\prodscal{a\otimes f}{b\otimes g}_{\sD(B_q)})\cdot 1\\
   &=(\pi_T(a^** b)\otimes \pi_T(f^**g)K_{-2\rho})\cdot 1\\
   &=\phi_T(\pi_T(a^** b) e^{-\mu})\phi_{\widehat{K}_q}((f^**g)K_{-2\rho-\lambda})\\
   &=\phi_T(\pi_T(a^** b) e^{-\mu})\phi_{\widehat{K}_q}(f^*K_{-2\rho-\lambda})\phi_{\widehat{K}_q}(gK_{-2\rho-\lambda}),
\end{align*}
where at the last line we used that $\phi_{\widehat{K}_q}(x*y)=\phi_{\widehat{K}_q}(x)\phi_{\widehat{K}_q}(y)$, $\forall x,y\in \sD(\widehat{K}_q)$. Note also that $\phi_{\widehat{K}_q}(f^*K_{-2\rho-\lambda})=(K_{-2\rho-\lambda}^*f^*)$ since $K_{-2\rho-\lambda}$ is self-adjoint and $\sigma_{\widehat{K}_q}(K_{-2\rho-\lambda})=K_{-2\rho-\lambda}$. For the calculation on the right hand side we will use that $(e^\mu)^*=e^\mu$ and that $e^\mu*e^\mu=e^\mu$. We also use Lemma \ref{lem:dual_ip}. We have
\begin{align*}
    \prodscal{a*\hat{\pi}_T(e^{\mu})\otimes K_{\lambda+2\rho}}{b*\hat{\pi}_T(e^{\mu})\otimes K_{\lambda+2\rho}}&=\prodscal{a*\hat{\pi}_T(e^{\mu})}{b*\hat{\pi}_T(e^{\mu})}\\
    &=\epsilon_{K_q}((a*\hat{\pi}_T(e^{\mu}))^**b*\hat{\pi}_T(e^{\mu}))\\
    &=\epsilon_{T}(\pi_T(\hat{\pi}_T(e^{\mu})^**a^*))*(b*\hat{\pi}_T(e^{\mu}))\\
    &=\epsilon_{T}(e^{\mu}*\pi_T(a^**b)*e^{\mu})\\
    &=\epsilon_{T}(\pi_T(a^**b)*(e^{\mu})^*)\\
    &=\phi_{T}(\pi_T(a^**b)\overline{e^{\mu}})\\
    &=\phi_{T}(\pi_T(a^**b)e^{-\mu}).\\
\end{align*}
This then shows that $\Psi$ descends to an unitary map on the balanced tensor product  $\mathcal{D}(G_q)\otimes_{\sD(B_q)}\mathbb{C}_{\mu,\lambda}\to {\rm Ind}_{B_q}^{G_q}\mathbb{C}_{\mu,\lambda}$. 

To conclude, we show that $\Psi$ is surjective. To this end we first notice that $\mathrm{Ind}_{B_q}^{G_q}\mathbb{C}_{\mu,\lambda}$ is spanned by elements of type $a\otimes K_{\lambda+2\rho}$ for $a\in \Gamma(\mathcal{E}_{\mu,\lambda})$, where $\Gamma(\mathcal{E}_{\mu,\lambda})$ is defined in \cite[Section 6.4.2]{yuncken-voigt}. This follows from the fact that the map $\mathrm{ext} : \Gamma(\mathcal{E}_{\mu,\lambda})\to \mathrm{Ind}_{B_q}^{G_q}\mathbb{C}_{\mu,\lambda}$ from \cite[Lemma 6.18]{yuncken-voigt} is an isomorphism and we have $\mathrm{ext}(a)=a\otimes K_{\lambda+2\rho}$ for all $a\in \Gamma(\mathcal{E}_{\mu,\lambda})$. Let then $a\in \Gamma(\mathcal{E}_{\mu,\lambda})$. We have that $a*\hat{\pi}_T(e^{\mu})=a$; thus the element $a\otimes \epsilon_{K_q}\otimes 1$ of $\mathcal{D}(G_q)\otimes_{\sD(B_q)}\mathbb{C}_{\mu,\lambda}$ is an antecedent of $a\otimes K_{\lambda+2\rho}$.
%where $\widehat{1_{K_q}}$ is defined in Section \ref{sec:compact}.
\end{proof}

One can now consider the $\sD(L_q)$-inner product module $\sD(G_q)\otimes_{\sD(B_q)}\sD(L_q)$ and we have 
$$\sD(G_q)\otimes_{\sD(B_q)}\sD(L_q)\otimes_{\sD(L_q)}\mathbb{C}_{\mu,\lambda}\cong  \sD(G_q)\otimes_{\sD(B_q)}\mathbb{C}_{\mu,\lambda}.$$
As a consequence, $\sD(G_q)\otimes_{\sD(B_q)}\sD(L_q)$ is the parabolic induction module.

\subsection{{Geometric presentation of the induction module}}

We consider the linear space
$$\mathcal{A}(G_q/N_q)=\mathcal{A}(K_q) \otimes\mathcal{A}(A_q),$$
equipped with its natural structure of untwisted $*$-algebra. We endow $\mathcal{A}(G_q/N_q)$ with a left $\mathcal{A}(G_q)$-coaction given, for $a\otimes h \in\mathcal{A}(G_q/N_q)$, by
$$\Delta_{G_q/N_q}(a\otimes h)= W^{-1}_{32}(\Delta_{K_q}(a)\otimes\alpha_{A_q}(h))W_{32}\ \in M(\mathcal{A}(G_q))\otimes\mathcal{A}(G_q/N_q),$$
where the coation $\alpha_{A_q}$ is defined after Lemma \ref{lm:formula}.  Let $f\in \mathcal{A}(\widehat{K_q})$. We have
$$\Delta_{G_q/N_q}(a\otimes \pi_T(f))= W^{-1}_{32}(a_{(1)}\otimes f_{(1)}\otimes a_{(2)}\otimes \pi_T(f_{(2)}))W_{32}.$$
From this we see that $\Delta_{G_q/N_q}(a\otimes \pi_T(f))=(\id\otimes\id\otimes\id\otimes\pi_T)(\Delta_{G_q}(a\otimes f))$ and it directly follows that the map $\Delta_{G_q/N_q}$ is coassociative.
This remark also implies the next proposition.

%This clearly defines a coaction since $\hat{\pi}_T : \mathcal{A}(A_q)\to \sA(\widehat{K_q})$ is a *-Hopf morphism. One can also observe that $\Delta_{G_q/N_q}$ is an *-algebra homorphism as a composition of *-algebra homorphisms. , we have

\begin{prop}\label{intertwiner}
The map $\id\otimes \pi_T : \mathcal{A}(G_q)\rightarrow \mathcal{A}(G_q/N_q)$ intertwines the left-$\mathcal{A}(G_q)$-coactions where $\mathcal{A}(G_q)$ is considered with its natural comodule structure given by the coproduct. 
\end{prop}

We now define a right $\mathcal{A}(L_q)$-coaction on $\mathcal{A}(G_q/N_q)$, denoted $\Delta'_{G_q/N_q}$. For  $a\otimes h\in \mathcal{A}(G_q/N_q)$ we set
$$\Delta'_{G_q/N_q}(a\otimes h)= a_{(1)}\otimes h_{(1)}\otimes \pi_T(a_{(2)})\otimes h_{(2)}\ \in \mathcal{A}(G_q/N_q)\otimes \mathcal{A}(L_q).$$

\begin{prop}\label{prop:commute}
The coactions $\Delta'_{G_q/N_q}$ and $\Delta_{G_q/N_q}$ commute.
\end{prop}
\begin{proof}
We first claim that we have 
$$(\id\otimes \hat{\pi}_T\otimes \id\otimes\hat{\pi}_T)(\Delta'_{G_q/N_q}(a\otimes h))=(\id\otimes \id\otimes\pi_T\otimes \id)[\Delta_{G_q}(a\otimes \hat{\pi}_T(h))].$$

We calculate
\begin{align*}
    (\id\otimes \id&\otimes\pi_T\otimes \id)[\Delta_{G_q}(a\otimes \hat{\pi}_T(h))]
    \\&=a_{(1)}\otimes \omega_{ij}^{\sigma}\hat{\pi}_T(h_{(1)})\omega_{rs}^{\nu}\otimes \pi_T(S(u_{ij}^{\sigma}) a_{(2)}u_{rs}^{\nu})\otimes \hat{\pi}_T(h_{(2)})\\
    &=a_{(1)}\otimes \hat{\pi}_T(h_{(1)})\otimes \pi_T(a_{(2)})\otimes \hat{\pi}_T(h_{(2)})\\
    &=(\id\otimes \hat{\pi}_T\otimes \id\otimes\hat{\pi}_T)(\Delta'_{G_q/N_q}(a\otimes h)).
\end{align*}

We have also that
$$(\id\otimes \id\otimes\id\otimes \hat{\pi})(\Delta_{G_q/N_q}(a\otimes h))=\Delta_{G_q}(a\otimes \hat{\pi}(h)).$$

Now we can prove the proposition. First we rewrite above equalities using the leg notation (we will write $\pi$ and $\hat\pi$ instead of $\pi_T$ and $\hat\pi_T$): 
\begin{align*}
(\hat{\pi}\otimes \hat{\pi})_{24}\circ\Delta'_{G_q/N_q}&=\pi_3\circ\Delta_{G_q}\circ \hat{\pi}_2\\ \hat{\pi}_4\circ\Delta_{G_q/N_q}&=\Delta_{G_q}\circ \hat{\pi}_2\end{align*}
Now observe that we have on the one hand 
\begin{align*}
    (\hat{\pi}\otimes\hat{\pi})_{46}\circ&(\Delta'_{G_q/N_q})_{34}\circ\Delta_{G_q/N_q}\\
    &=((\hat{\pi}\otimes\hat{\pi})_{24}\circ \Delta'_{G_q/N_q})_{34}\circ\Delta_{G_q/N_q}\\
    &=(\pi_3\circ \Delta_{G_q}\circ \hat{\pi}_2)_{34}\circ\Delta_{G_q/N_q}\\
    &=\pi_5\circ (\Delta_{G_q})_{34}\circ \Delta_{G_q}\circ \hat{\pi}_2
\end{align*}
and on the other hand 
\begin{align*}
    (\hat{\pi}\otimes\hat{\pi})_{46}\circ&(\Delta_{G_q/N_q})_{12}\circ\Delta'_{G_q/N_q}\\
&=(\Delta_{G_q})_{12}\circ(\hat{\pi}\otimes\hat{\pi})_{24}\circ \Delta'_{G_q/N_q}\\
&=(\Delta_{G_q})_{12}\circ \pi_3\circ \Delta_{G_q}\circ \hat{\pi}_2\\
&=\pi_5\circ (\Delta_{G_q})_{12}\circ \Delta_{G_q}\circ \hat{\pi}_2\\
\end{align*}
and we conclude the proof using the coassociativity of $\Delta_{G_q}$ and injectivity of $\hat{\pi}_T$.
\end{proof}

Observe now that $\mathcal{A}(G_q/N_q)=\sD(K_q)\otimes\sD(A_q)$ as linear space. On the one hand $\sD(K_q)$ can be considered as a $\sD(T)$-inner product module, since $T$ is a closed quantum subgroup of $K_q$. On the other hand $K_{2\rho}$ is a self-adjoint and group-like element of  $M(\sA(A_q))$; thus $\sD(A_q)$ has a structure of $\sD(A_q)$-inner product module with right action
$$h\cdot l=h*(lK_{2\rho}), $$
and the sesquilinear map defined by
$$\prodscal{h}{k}_{\sD(A_q)}=(h^**k)K_{-2\rho},$$
for all $h,k,l\in \sD(A_q)$. One can thus endow $\mathcal{A}(G_q/N_q)=\sD(K_q)\otimes\sD(A_q)$ with the structure of a $(\sD(T)\otimes \sD(A_q))$-inner product module induced by the tensor product. Let $a\otimes h,~ b\otimes k\in \mathcal{A}(G_q/N_q)$ and $\tau\otimes l\in \sD(L_q)$. We have
\begin{align*}
    \prodscal{a\otimes h}{b\otimes k}_{\sD(L_q)}&=\pi_T(a^**b)\otimes (h^**k)K_{-2\rho},\\
    (a\otimes k)\cdot(\tau\otimes l)&=a*\hat\pi_T(\tau)\otimes k*(lK_{2\rho}).
\end{align*}
\begin{lm}
The left action of $\sD(K_q)$ on $\mathcal{A}(G_q/N_q)$ induced by $\Delta_{G_q/N_q}$ commutes with the right $\sD(L_q)$ action.
\end{lm}
\begin{proof}
This is almost equivalent to Proposition \ref{prop:commute}. Observe that if one precomposes the right $\sD(L_q)$ action by the *-algebra homomorphism of $\sD(L_q)$ given by $x\mapsto x(1\otimes K_{2\rho})$ we obtain exactly the action induced by the coaction $\Delta'_{G_q/N_q}$.\end{proof}
\begin{prop}\label{prop:intertwiner}
The map defined by  
\begin{align*}
    \Phi~:~\mathcal{D}(G_q)\otimes\sD(L_q)&\longrightarrow \mathcal{A}(G_q/N_q)\\
    (a\otimes f)\otimes (\tau\otimes h) &\longmapsto (a\otimes \pi_T(f))\cdot (\tau\otimes h),
\end{align*}
is a $\sD(L_q)$-linear map which intertwines the left action of $\mathcal{D}(G_q)$ and descends to a unitary isomorphism on the balanced tensor product $\mathcal{D}(G_q)\otimes_{\sD(B_q)}\sD(L_q)$.
\end{prop}

\begin{proof}
The $\sD(L_q)$-linearity of $\Phi$ is immediate from the definition since the right $\sD(L_q)$ action on $\mathcal{A}(G_q/N_q)$ is associative. The intertwinning property directly follows from Proposition \ref{intertwiner} and the previous proposition. Let $(a\otimes f)\otimes (\tau\otimes h)$ and $(b\otimes g)\otimes (\zeta\otimes k)$ be in $\mathcal{D}(G_q)\otimes\sD(L_q)$ and consider the elements $[(a\otimes f)\otimes (\tau\otimes h)]$ and $[(b\otimes g)\otimes (\zeta\otimes k)]$ of the balanced tensor product $\mathcal{D}(G_q)\otimes_{\sD(B_q)}\sD(L_q)$. We have 
\begin{align*}
    &\prodscal{[(a\otimes f)\otimes (\tau\otimes h)]}{[(b\otimes g)\otimes (\zeta\otimes k)]}_{\sD(L_q)}\\
    &~~~~~~~~~~~~~~~=\prodscal{(\tau\otimes h)}{\prodscal{a\otimes f}{b\otimes g}_{\sD(B_q)}\cdot (\zeta\otimes k)}_{\sD(L_q)}\\
    &~~~~~~~~~~~~~~~=\prodscal{(\tau\otimes h)}{(\pi_T(a^**b)\otimes \pi_T(f^**g)K_{-2\rho})*(\zeta\otimes k)}_{\sD(L_q)}\\
    &~~~~~~~~~~~~~~~=\prodscal{(\tau\otimes h)}{(\pi_T(a^**b)*\zeta)\otimes( \pi_T(f^**g)K_{-2\rho})*k)}_{\sD(L_q)}\\
    &~~~~~~~~~~~~~~~=(\tau^**\pi_T(a^**b)*\zeta) \otimes (h^**\pi_T((f^**g)K_{-2\rho})*k)\\
    &~~~~~~~~~~~~~~~=(\pi_T(a*\hat{\pi}_T(\tau))^**b*\hat{\pi}_T(\zeta)) \otimes (h^**\pi_T(f^**g)K_{-2\rho}*k)\\
    &~~~~~~~~~~~~~~~=(\pi_T(a*\hat{\pi}_T(\tau))^**b*\hat{\pi}_T(\zeta)) \otimes (h^**\pi_T(f)^*K_{2\rho}*\pi_T(g)K_{2\rho}*k)K_{-2\rho}\\
    &~~~~~~~~~~~~~~~=\prodscal{(a\otimes \pi_T(f))\cdot (\tau\otimes h)}{(b\otimes \pi_T(g))\cdot (\zeta\otimes k)}_{\sD(L_q)}.
\end{align*}
Thus the map $\Psi$ descend to a unitary map on the balanced tensor product. With regard to the surjectivity it is enough to observe that the right $\sD(L_q)$-action on $\mathcal{A}(G_q/N_q)$ is essential.
\end{proof}
The following theorem is now immediate.
\begin{thm}\label{iso}
The pre-Hilbert $\mathcal{D}(L_q)$-module  $\mathcal{A}(G_q/N_q)$ can be completed into a Hilbert $C^*(L_q)$-module $\mathcal{E}(G_q/N_q)$ and we have 
$$\mathcal{E}(G_q/N_q)\cong \mathcal{E}(G_q)\otimes_{C^*(B_q)} C^*(L_q),$$ 
as $G_q$-representations. The tensor product $\mathcal{E}(G_q/N_q)\otimes_{C^*(L_q)}\--$ defines a functor from the category of unitary $C^*(L_q)$-representations to the category of unitary $C^*_u(G_q)$-representations which coincides with parabolic induction.
\end{thm}
By the Fourier transform, we have 
\[C^*(L_q)\cong C_0(\widehat{L_q})=C_0(\textbf{P}\times T),\]
such that the characters of $C^*(L_q)$ become the evaluation maps 
\[\mathrm{ev}_{(\mu,\lambda)}~:~C_0(\textbf{P}\times T)\rightarrow \mathbb{C}_{\mu,\lambda}.\]

 %Theorem \ref{iso} shows that  the Hilbert space bundle of principal series representations 
%$\mathcal{H}=(\text{Ind}_{B_q}^{G_q}\mathbb{C}_{\mu,\lambda})_{\mu,\lambda}$
%over $\textbf{P}\times T$ is isomorphic to the $C_0(\textbf{P}\times T)$-Hilbert module 
%\[\mathcal{E}(G_q/N_q)\otimes_{C_0(\textbf{P}\times T)}C_0(\textbf{P}\times T)\cong \mathcal{E}(G_q/N_q).\]
%According to \cite{yuncken-voigt}, the action of the Weyl group $W$ on $\textbf{P}\times T$ extends to a action by $C^*_u(G_q)$-linear maps on $\mathcal{H}\cong \mathcal{E}(G_q/N_q)$ and then Theorem 7.1 of \cite{Plancherel} can be can be interpreted as follows.

According to \cite[Theorem 7.1]{yuncken-voigt} we have 
$$C^*_r(G_q)\cong C_0(\textbf{P}\times \mathfrak{t}_q^*,\mathcal{K}(H))^W,$$
where $H$ is a countable dimensional Hilbert space, and the action of the Weyl group $W$ is a lifting of its action by reflections on $\textbf{P}\times \mathfrak{t}_q^*$ to an action on the bundle of $C^*$-algebras. More precisely, the Hilbert space $H$ at the parameter $(\mu,\lambda)\in \textbf{P}\times \mathfrak{t}_q^*$ is identified with the parabolically induced representation of $G_q$,
\begin{align*}
    H=&H_{\mu,\lambda}=\mathrm{Ind}_{B_q}^{G_q}\CC_{\mu,\lambda}\\
    \cong&H_{\mu}=\mathrm{Ind}_{T}^{K_q}\CC_{\mu}\\
    =&\overline{\{\xi\in\sA(K_q)~|~\Delta(\xi)=\xi\otimes e^\mu\}}^{\|\cdot\|_{L^2(K_q)}},
\end{align*}
which is a trivial Hilbert bundle on each connected component $\{\mu\}\times \mathfrak{t}_q^*$ of the parameter space. The action of $W$ is via intertwiners of principal series representations. In this way, we have 
$$C^*_r(G_q)=(\mathfrak{K}(\underset{\mu\in \textbf{P}}{\bigoplus}C_0(\mathfrak{t}_q^*,H_\mu)))^W, $$
where $\mathfrak{K}$ denotes compact operators on the right Hilbert $C_0(\textbf{P}\times \mathfrak{t}_q^*)$-module. 

By theorem \ref{iso} we have 
\[H_{\mu,\lambda}\cong  \mathcal{E}(G_q)\otimes_{C^*(B_q)}\CC_{\mu,\lambda},\]
as left $C^*_u(G_q)$-module. Therefore 
$$C_0(\mathfrak{t}_q^*,H_\mu)\cong  \mathcal{E}(G_q/N_q)\otimes_{C^*(L_q)}C_0(\mathfrak{t}_q^*)_\mu,$$
as left $C^*_u(G_q)$-module and right $C_0(\mathfrak{t}_q^*)$-Hilbert module, where $C_0(\mathfrak{t}_q^*)_\mu$ denotes $C_0(\mathfrak{t}_q^*)$ equipped with the left action of $C^*(L_q)=C^*(T)\hat\otimes C^*(A_q)=C_0(\textbf{P})\hat\otimes C_0(\mathfrak{t}_q^*)$ such that $C_0(\mathfrak{t}_q^*)$ acts by pointwise multiplication and $C_0(\textbf{P})$ acts by evaluation at $\mu$. We thus obtain 
\begin{align*}
   \underset{\mu\in \textbf{P}}{\bigoplus}C_0(\mathfrak{t}_q^*,H_\mu)&=\mathcal{E}(G_q/N_q)\otimes_{C^*(L_q)}C_0(\textbf{P}\times \mathfrak{t}_q^*)\\
   &=\mathcal{E}(G_q/N_q).
\end{align*}
We have therefore proven the following result.

\begin{cor}
Let $G_q$ be a complex semi-simple quantum group. Then 
\[C^*_r(G_q)\cong \mathfrak{K}(\mathcal{E}(G_q/N_q))^W,\]
where $\mathfrak{K}$ indicates the algebra of compact operators in the sense of Hilbert modules.
\end{cor}

In the classical case, this result has been first obtained in \cite{Wassermann} and reformulated in \cite{CCH} with the Rieffel induction framework.

\bibliographystyle{alpha}
\bibliography{refs.bib}

\end{document}